\documentclass[11pt,reqno]{amsart}

\numberwithin{equation}{section}
\usepackage{times}
\usepackage{amsmath,amsfonts,amstext,amssymb,amsbsy,amsopn,amsthm,eucal}
\usepackage{txfonts}
\usepackage{dsfont}
\usepackage{graphicx}   
\usepackage{hyperref}
\usepackage{accents}
\usepackage{enumerate}
\usepackage{xcolor}
\usepackage{verbatim}
\usepackage{esint}

\usepackage[normalem]{ulem}
\usepackage{cancel}

\newcommand{\RR}{\mathds{R}}

\newcommand{\Vol}{{\rm Vol}}

\newcommand{\cA}{\mathcal{A}}

\newcommand{\cC}{\mathcal{C}}

\newcommand{\cH}{\mathcal{H}}

\newcommand{\cL}{\mathcal{L}}

\newcommand{\cN}{\mathcal{N}}

\newcommand{\cP}{\mathcal{P}}

\newcommand{\cS}{\mathcal{S}}

\newcommand{\cV}{\mathcal{V}}

\newcommand{\cZ}{\mathcal{Z}}

\newtheorem{theorem}[equation]{Theorem}

\newtheorem{proposition}[equation]{Proposition}
\newtheorem{lemma}[equation]{Lemma}
\newtheorem{corollary}[equation]{Corollary}
\theoremstyle{definition}
\newtheorem{definition}[equation]{Definition}

\theoremstyle{remark}
\newtheorem{remark}[equation]{Remark}
\theoremstyle{remark}

\theoremstyle{remark}

\theoremstyle{remark}
\theoremstyle{remark}

\begin{document}

\thanks{}
\thanks{}

\title[Volume Estimates for Singular sets and Critical Sets of Elliptic Equations with Hölder Coefficients]{Volume Estimates for Singular sets and Critical Sets of Elliptic Equations with Hölder Coefficients}

\author{Yiqi Huang}
\address[Yiqi Huang]{Department of Mathematics, MIT, 77 Massachusetts Avenue, Cambridge, MA 02139-4307, USA}
 \email{yiqih777@mit.edu}

\author{Wenshuai Jiang}
\address[Wenshuai Jiang]{School of Mathematical Sciences, Zhejiang University, Hangzhou 310058, China}
 \email{wsjiang@zju.edu}

\maketitle

\begin{abstract}
Consider the solutions $u$ to the elliptic equation $\cL(u) = \partial_i(a^{ij}(x) \partial_j u) + b^i(x) \partial_i u + c(x) u= 0$ with $a^{ij}$ assumed only to be H\"older continuous. In this paper we prove an explicit bound for $(n-2)$-dimensional Minkowski estimates of singular set $\cS(u) = \{ x \in B_1 : u(x) = |\nabla u(x)| = 0\}$ and critical set $\cC(u) \equiv \{ x\in B_{1} : |\nabla u(x)| = 0 \}$ in terms of the bound on doubling index, depending on $c \equiv 0$ or not. This generalizes the Lipschitz assumption in \cite{NV} to H\"older. It is sharp as it is the weakest condition in order to define the critical set of $u$ according to elliptic estimates. We can also obtain an optimal improvement on Cheeger-Naber-Valtorta's volume estimates on each quantitative stratum $\cS^k_{\eta, r}$ as in \cite{CNV}. The main difficulty in this situation is the lack of monotonicity formula which is essential to the quantitative stratification. In our proof, one key ingredient is a new almost monotonicity formula for doubling index under the H\"older assumption.  Another key ingredient is the quantitative uniqueness of tangent maps. It deserves to note that our almost monotonicity is sufficient to address all the difficulties arising from the absence of monotonicity in the analysis of differential equations. We believe the idea could be applied to other relevant study. 
\end{abstract}

\tableofcontents

\section{Introduction}\label{s:intro}

In this paper we consider the solutions to the following elliptic equation in divergence form on $B_2 \subset \RR^n$
\begin{equation}\label{e:general_elliptic_equ}
    \cL(u) = \partial_i(a^{ij}(x) \partial_j u) + b^i(x) \partial_i u + c(x) u= 0,
\end{equation}
where the coefficients $a^{ij}$ are elliptic and merely assumed to be in $C^{\alpha}({B_2})$ with $\alpha \in (0, 1)$, and the coefficients $b, c$ are bounded,
\begin{equation}\label{e:assumption_with_c}
    (1 + \lambda)^{-1} \delta^{ij} \leq a^{ij} \leq (1 + \lambda) \delta^{ij}, \quad |a^{ij}(y) -a^{ij}(z)| \leq \lambda|y-z|^{\alpha}, \quad |b^i|, |c| \leq \lambda.
\end{equation}

The nodal set is defined to be the zero set of $u$, i.e. $\cZ(u) \equiv \{ x \in B_1: u(x) = 0\}$. The critical set and singular set of $u$ are defined as
\begin{equation}
    \cC(u) \equiv \{ x\in B_{1} : |\nabla u(x)| = 0 \}, \quad \text{ and } \quad \cS(u) = \{ x \in B_1 : u(x) = |\nabla u(x)| = 0\}.
\end{equation}

The main goal of this paper is to prove the volume estimates on the critical sets when $c\equiv 0$ and the singular set for general equations.

Over the past several decades, there has been extensive literature about the nodal sets and the singular sets of solutions to elliptic equations. In \cite{DF1}, Donnelly and Fefferman proved the optimal upper and lower bounds for Hausdorff measure of nodal sets of Laplacian eigenfunctions on compact manifolds under the assumption that the manifold is analytic. This confirmed Yau's conjecture \cite{Yau} in the analytic setting.  There are many interesting discussions towards Yau's conjecture by Br\"uning\cite{Bru}, Colding-Minicozzi \cite{CM}, Dong \cite{Dong}, Donnelly-Fefferman \cite{DF2} \cite{DF3}, Nadirashvili \cite{N88} \cite{N}.  More recently, a major breakthrough was made by Logunov \cite{Loupper} \cite{Lolower} who proved the sharp lower bound and obtained polynomial upper bound without analyticity assumption.
 
Concerning the nodal sets of solutions to general elliptic equations, in \cite{Linconj}, Lin proved the optimal upper bound for Hausdorff measure of nodal sets of solutions to second order elliptic equations with analytic coefficients. For the nonanalytic case, Hardt and Simon \cite{HS} obtained an explicit (though not optimal) estimate on the Hausdorff measure of nodal sets in terms of the growth of the solutions with the coefficients being only H\"older continuous (See also \cite{HLJPDE}). See also \cite{BeLin} \cite{Hanhamornic} \cite{HLparobolic}  \cite{HHLhigher}  \cite{HSo} \cite{Linconj} \cite{LLbetti} \cite{LS1} \cite{LS2} \cite{LZhu} \cite{Man} \cite{Mil} \cite{SWZ} \cite{SZ} \cite{Ste} \cite{TZ} \cite{Zhu2} for more relevant study.

According to the implicit function theorem, the nodal set of a smooth function is a smooth hypersurface away from the singular set. Hence it is crucial to study the singular set in order to understand the structure of nodal sets. There is rich study on singular sets of solutions to elliptic equations. In \cite{HLbook} \cite{HHL}  \cite{HHHN} \cite{HHN}, they obtained a bound of the $(n-2)$-dimensional Hausdorff measure under the assumption of smooth coefficients in terms of the frequency. The same result was proved by Cheeger, Naber and Valtorta in \cite{CNV} where they also improved the Hausdorff estimates to Minkowski estimates. In \cite{NV}, Naber and Valtorta proved the bound only assuming the leading coefficients are Lipschitz. 

By elliptic estimates, if $a^{ij} $ is assumed $C^{\alpha}$, the weak solution to \ref{e:general_elliptic_equ} is $C^{1,\alpha}$ and the singular set is still well-defined. Hence it is a natural question to ask if one could obtain the estimates on the singular set with $a^{ij}$ being only H\"older. Because of the failure of unique continuation under H\"older assumption \cite{P}, we need to exclude the enemies which vanish in some open ball as it makes no sense to control the nodal sets or singular sets in these cases. It turns out that being constant or close to constant are all bad cases. In this paper, we will give an affirmative answer to this question. This is optimal as under the $C^{0}$ assumption on $a^{ij}$, the weak solution $u$ is only $C^{\alpha}$ and thus the gradient is not even pointwisely well-defined. 

We will use the doubling index, see Definition \ref{d:doubling_index}, to measure the how far the solutions are  away from being constant. Since the frequency is no more monotone when $a^{ij}$ is only H\"older, the following doubling assumption on elliptic solutions $u$ is the correct growth assumption one should make:
there exists some positive constant $\Lambda$ such that 
\begin{equation}\label{e:Singular_DI_bound}
    \sup_{B_{2r}(x) \subset B_2} \text{log}_4 \frac{\fint_{ B_{2r}(x)} u^2}{\fint_{B_{r}(x)} u^2}  \leq \Lambda.    
\end{equation}

Note that this assumption is automatically satisfied if $a^{ij}$ is Lipschitz. See Theorem 2.28 in \cite{HLbook}, Theorem 3.14 and Theorem 4.6 in \cite{NV}. In this case, the bound in \ref{e:Singular_DI_bound} linearly depends on the frequency in $B_1$. 

\subsection{Results for Singular sets and Critical sets}

Our first main theorem is the following Minkowski estimates for singular sets.

\begin{theorem}\label{t:singular_set}
Let $u: B_2 \to \RR$ be a nonzero solution to (\ref{e:general_elliptic_equ}) (\ref{e:assumption_with_c}) with doubling assumption \ref{e:Singular_DI_bound}. Then for any $0<r\leq 1$ we have
\begin{equation}
    \Vol (B_r(\cS(u)) \cap B_{1}) \leq C(n, \lambda, \alpha)^{\Lambda^2} r^2.
\end{equation}
In particular, we have the estimate for Minkowski dimension $\dim_{\text{Min}} \cS(u) \leq n-2$.
\end{theorem}

Note that the Minkowski estimates are much stronger than Hausdorff estimates. For instance, the set of rational numbers in $\RR^n$ are of 0-Hausdorff dimension but of $n$-Minkowski dimension.

Next we study the estimates on critical sets of solutions with $c \equiv 0$. 
\begin{equation}\label{e:elliptic_equ}
    \cL(u) = \partial_i(a^{ij}(x) \partial_j u) + b^i(x) \partial_i u= 0,
\end{equation}
where the coefficients $a^{ij}$ are elliptic and in $C^{\alpha}({B_2})$, and the coefficients $b$ is bounded,
\begin{equation}\label{e:assumption}
    (1 + \lambda)^{-1} \delta^{ij} \leq a^{ij} \leq (1 + \lambda) \delta^{ij}, \quad |a^{ij}(y) -a^{ij}(z)| \leq \lambda|y-z|^{\alpha}, \quad |b^i| \leq \lambda.
\end{equation}

In this case we impose the following normalized doubling assumption for any 
\begin{equation}\label{e:DI_bound_assumption}
    \sup_{B_{2r}(x)\subset B_2}\text{log}_4 \frac{\fint_{ B_{2r}(x)} (u - u(x))^2}{\fint_{B_{r}(x)} (u - u(x))^2}  \leq \Lambda.    
\end{equation}

\begin{theorem}\label{t:critical_set}
Let $u: B_2 \to \RR$ be a nonconstant solution to (\ref{e:elliptic_equ}) (\ref{e:assumption}) and doubling assumption \ref{e:DI_bound_assumption}. Then for any $0<r\leq 1$ we have
\begin{equation}
    \Vol (B_r(\cC(u)) \cap B_{1}) \leq C(n, \lambda, \alpha)^{\Lambda^2} r^2.
\end{equation}
In particular, we have the estimate for Minkowski dimension $\dim_{\text{Min}} \cC(u) \leq n-2$.
\end{theorem}

Note that Theorem \ref{t:critical_set} does not hold
for general ellipitc equations \ref{e:general_elliptic_equ} with $c \not\equiv 0$. For example, consider $u = v^2 + 1$ with some $v \in C^{\infty}(B_2)$ and $|v| < 1$. Then $u$ satisfies $\Delta u + cu = 0$ with $c = - \Delta v^2 (v^2 + 1)^{-1}$. Note that $\cZ(v) \subset \cC(u)$ while for any closed subset $K$ in $\RR^n$ there exists some smooth function $v$ vanishing on $K$. See \cite{HHHN}.

The proof of Theorem \ref{t:singular_set} and Theorem \ref{t:critical_set} are almost the same up to very little modification. Throughout this paper, we will focus on the equation (\ref{e:elliptic_equ}) (\ref{e:assumption}) and assume the doubling assumption \ref{e:DI_bound_assumption}.

 \subsection{Results for Quantitative Stratification}   

Next we will introduce the quantitative stratification which was first proposed in \cite{CN} \cite{CNV}. The stratification separates points based on the tangential behavior of the functions, more precisely the number of symmetry of the leading homogeneous polynomial of the Taylor expansion of the function. In \cite{CNV}, the stratification was refined in a more quantitative way and they obtain an effective Minknowski estimate on the quantitative strata. Our next result is the optimal improvement on these Minkowski estimates on each quantitative stratum. 

First we define the scaling map and tangent map. For any $x$, we define the linear transformation $A_x$ by 
\begin{equation*}
    A_x(y) = (\sqrt{a})^{ij} y_i e_j
\end{equation*}
where $(\sqrt{a})^{ij}$ is the square root of the leading coefficients matrix $a^{ij}(x)$. Note that for any $x$, $A_x$ is a Lipschitz equivalence with Lipschitz constant $(1 + \lambda)^{1/2}$. 
With this we can define the rescaling map of the solution $u$ to \ref{e:elliptic_equ}, \ref{e:assumption}. 

\begin{definition}
Let $u: B_2 \to \RR$ be a solution to \ref{e:elliptic_equ} \ref{e:assumption} with doubling assumption \ref{e:DI_bound_assumption}. Let $r \leq \frac{1}{(1+  \lambda)^{1/2}}$.
\begin{enumerate}
    \item For $x \in B_{1}$, we define
    \begin{equation*}
       u_{x, r}(y) := \frac{u(x +rA_x(y)) - u(x)}{\Big( \fint_{\partial B_1} (u(x +rA_x(z)) - u(x))^2 dz\Big)^{1/2}}.
    \end{equation*}
    \item For $x\in B_{1}$, we define
    \begin{equation*}
        u_{x}(y) = \lim_{r \to 0} u_{x,r}(y). 
    \end{equation*}
\end{enumerate}
\end{definition}

\begin{remark}\label{r:Rescaled_u_equation}
Denote $\Tilde{u} = u_{x, r}$. Then by change of variables, $\Tilde{u}$ satisfies the equation
\begin{equation}\label{e:Rescaled_u_equation}
    \Tilde{\cL}(\Tilde{u}) = \partial_i(\Tilde{a}^{ij} \partial_j \Tilde{u}) + \Tilde{b}^i \partial_i \Tilde{u} = 0,
\end{equation}
where $\Tilde{a}^{ij}(0) = \delta^{ij}$ and $\tilde{a}(y)=a(x+rA_x(y)) a(x)^{-1} $ and $\tilde{b}(y)= r b(x+rA_x(y))\sqrt{a}(x)^{-1}$ with 

\begin{equation}
    (1 + C(\lambda,\alpha) r^{\alpha})^{-1} \delta^{ij} \leq \Tilde{a}^{ij} \leq (1 + C(\lambda,\alpha)) r^{\alpha}) \delta^{ij}, \quad |\Tilde{a}^{ij}(y) -\Tilde{a}^{ij}(z) | \leq C(\lambda,\alpha) r^{\alpha} |y -z|^{\alpha}, \quad |\tilde{b}^i| \leq r (1+\lambda)^{-1/2}.
\end{equation}

\end{remark}
In fact, this is a natural generalization of rescaling maps on Euclidean balls to the geodesic balls on general Riemannian manifolds.

Next we define the symmetry of functions.
\begin{definition}
Consider the continuous function $u: \RR^n \to \RR$.
\begin{enumerate}
    \item $u$ is called $0$-symmetric if $u$ is a homogeneous polynomial.
    \item $u$ is called $k$-symmetric if $u$ is $0$-symmetric and further symmetic with respect to some $k$-dimensional subspace $V$, i.e. $u(x+y) = u(x)$ for any $x \in \RR^n$ and $y \in V$.
\end{enumerate}
\end{definition}

\begin{definition}
Given a continuous function $u: B_1 \to \RR$. The singular $k$-stratum of $u$ is defined by
\begin{equation}
    \cS^k(u) \equiv \{x \in B_1 : T_x u \text{ is not } (k+1)\text{-symmetric } \}. 
\end{equation}
\end{definition}

Next we define the quantitative symmetry for $u$.
\begin{definition}\label{d:quan_sym}
Let $u: B_1 \to \RR$ be an $L^2$ function. We define $u$ is $(k, \eta, r, x)$-symmetric if there exists a $k$-symmetric polynomial $P_{x,r}$ with $\fint_{ \partial B_1} |P_{x,r}|^2 = 1$ such that 
\begin{equation}
    \fint_{\partial B_1} |u_{x, r} - P_{x,r}|^2 \leq \eta.
\end{equation}
\end{definition}

With this we can now give the definition of quantitative stratification
\begin{definition}
    Let $u: B_1 \to \RR$ be an $L^2$ function. The $(k,\eta,r)$-singular stratum is defined by 
    \begin{equation}
        \cS^k_{\eta,r} \equiv \{ x\in B_1 : u \text{ is not } (k+1, \eta, s, x)\text{-symmetric for any } s \geq r \}.
    \end{equation}
and we let
    \begin{equation}
        \cS^k_{\eta} \equiv \bigcap_{r}  \cS^k_{\eta, r}. 
    \end{equation}
\end{definition}
\begin{remark}
    The following observations are direct. 
    \begin{equation}
        \cS^k_{\eta, r} \subset \cS^{k'}_{\eta', r'} \text{ provided } k \leq k', \eta' \leq \eta, r \leq r'.
    \end{equation}
    \begin{equation}
        \cS^k = \bigcup_{\eta}\cS^k_{\eta} = \bigcup_{\eta} \bigcap_{r}  \cS^k_{\eta, r}. 
    \end{equation}
\end{remark}

The following Minkowski type estimate is an optimal improvement of Theorem 1.10 in \cite{CNV}. 

\begin{theorem}\label{t:Estimate_on_Strata}
Let $u: B_2 \to \RR$ be a solution to (\ref{e:elliptic_equ}) (\ref{e:assumption}) with doubling assumption \ref{e:DI_bound_assumption}. Then for any $k \leq n-2$ and $r \leq r_0(n,k,\lambda, \alpha, \Lambda)$, we have
\begin{equation}
    \Vol (B_r(\cS^k_{\eta, r}) \cap B_{1} ) \leq C(n, \lambda, \alpha, \Lambda, \eta) r^{n - k }.
\end{equation}
\end{theorem}

\subsection{Organization of the Paper}
In section 2, we first define the doubling index for solutions to general elliptic equations with $a^{ij}$ being only H\"older and then prove the almost monotonicity formula. The idea will be to approximate the solution $u$ with harmonic function $h$ and then to use the monotonicity of frequency of $h$ to obtain the almost monotonicity of doubling index of $u$. 

In section 3, we prove the quantitative uniqueness of tangent maps of $u$, which is crucial in the inductive decomposition arguments. The proof relies on the growth estimates for gradient of $u$ and the Green's function expansion. We will further prove more important properties of doubling index in this section. In particular, the doubling index only pinches near integers and it drops at a definite rate when away from integers.

In section 4, we discuss the cone splitting principle for harmonic functions and further general elliptic solutions. In section 5, we finish the proof of main theorems about the critical sets \ref{t:critical_set} and the singular sets \ref{t:singular_set}. In the last section, we prove the Minkowski estimates for quantitative strata \ref{t:Estimate_on_Strata}. The proof is based on the neck region decomposition technique.

\subsection*{Acknowledgements} Y. Huang is grateful to Prof. Tobias Colding for his invaluable encouragements and constant support. Y. Huang was supported by NSF DMS Grant 2104349. W. Jiang was supported by National Key Research and Development Program of China (No. 2022YFA1005501), National Natural Science Foundation of China (Grant Nos. 12125105 and 12071425) and the Fundamental Research Funds for the Central Universities K20220218.

\section{Almost Monotonicity Formula}

The main goal of this section is to prove the following almost monotonicity for doubling index.
\begin{theorem}
Let $u: B_2 \to \RR$ be  a solution to (\ref{e:elliptic_equ}) with (\ref{e:assumption}) with doubling assumption \ref{e:DI_bound_assumption}. For any $\epsilon \in (0, 1/10]$, there exists some $r_0 =\epsilon^{C(n,\lambda,\alpha)\Lambda}$ such that the following almost monotonicity formula holds for the doubling index:
\begin{equation}
    D(x,s) \leq D(x,r) + \epsilon,
\end{equation}
for any $x\in B_{1}$ and $ 0 < 2s \leq r \leq r_0$. In particular, we have $D(x,r) \geq 1 - \epsilon$ for any $0 < r \leq r_0$.
\end{theorem}

We will first give the formal definition of the doubling index $D$ for general elliptic solutions and then prove some fundamental properties of $D$. The proof of the main theorem could be found in section \ref{ss:Almost_mono}.

\subsection{Doubling Index for Harmonic Functions}

In this section, we collect some basic properties of doubling index for harmonic function. Recall that the doubling index for the nonconstant harmonic function $u$ is defined as
\begin{equation}
    D^{u}(x,r) = \text{log}_4 \frac{\fint_{ \partial B_{2r}(x)} (u - u(x))^2}{\fint_{\partial B_{r}(x)} (u - u(x))^2} .
\end{equation}
For general solution of elliptic equation, we define
\begin{definition}\label{d:doubling_index}
    Let $u: B_2 \to \RR$ be a solution to \ref{e:elliptic_equ} \ref{e:assumption} with doubling assumption \ref{e:DI_bound_assumption}. For any $x \in B_{1}, r \leq \frac{1}{(1+  \lambda)^{1/2}}$ , we define the doubling index as
    \begin{equation}
        {D}^u(x,r) \equiv D(x,r)\equiv \text{log}_4 \frac{\fint_{ \partial B_{2}(0)} (u(x +rA_x(y)) - u(x))^2 dy}{\fint_{ \partial B_1(0)} (u(x +rA_x(y)) - u(x))^2 dy}.
    \end{equation}
\end{definition}
\begin{remark}\label{r:Gen_DI}
    By scaling, it is straightforward that $D^u(x, rs) = D^{u_{x,r}} (0, s)$ for any $x \in B_{1}$, $s, r \leq \frac{1}{(1+  \lambda)^{1/2}}$.
\end{remark}

In particular, if $u$ is harmonic, then 
\begin{equation}
 D^{u}(x,r) = \text{log}_4 \frac{\fint_{ \partial B_{2r}(x)} (u - u(x))^2}{\fint_{\partial B_{r}(x)} (u - u(x))^2} ,
 \end{equation}
which coincides with the definition of the doubling index of harmonic function. We will omit the superscript for simplicity when there is no ambiguity.

The (normalized) Almgren frequency for nonconstant function $u$ is defined as follows
\begin{equation}
    N(x,r) \equiv \frac{r \int_{B_r(x)} |\nabla u|^2 }{\int_{\partial B_r(x)} (u - u(x))^2 }.
\end{equation}
This function is nondecresing in $r$.

Also we define the spherical average of $u$ as
\begin{equation}
    H(x, r) \equiv  \fint_{\partial B_r(x)} |u - u(x)|^2 
\end{equation}
The standard properties for $h$ is that for any $r_1 < r_2$, see \cite{GL} \cite{HLbook} \cite{LS1} \cite{LS2}.
\begin{equation}
    H(x, r_2) = H(x, r_1) \exp \big(2 \int^{r_2}_{r_1} \frac{N(x, s)}{s} ds\big).
\end{equation}

As a corollary, we have the following standard properties for harmonic functions. We include the proof here for the sake of completeness.
\begin{lemma}\label{l:Harmonic_AlmostMonotone}
    Let $u$ be a harmonic function in $B_1$. Then for any $0 < r_1 < r_2 < 1$ we have
    \begin{enumerate}
    \item $H(r)$ in nondecreasing in $r$. Moreover, we have $\Big(\frac{r_2}{r_1}\Big)^{2N(x, r_1)} \leq \frac{H(x, r_2)}{H(x, r_1)} \leq \Big(\frac{r_2}{r_1}\Big)^{2N(x, r_2)} .$
        \item $\Big(\frac{r_2}{r_1}\Big)^{2N(x, 0)} \leq \frac{\fint_{B_{r_2}} |u -u(x)|^2}{\fint_{B_{r_1}} |u - u(x)|^2} \leq \Big(\frac{r_2}{r_1}\Big)^{2N(x, r_2) }$ .
        \item $\frac{1}{n + 2D(x,r)} \fint_{\partial B_r} |u - u(x)|^2 \leq \fint_{B_r}  |u - u(x)|^2  \leq \frac{1}{n} \fint_{\partial B_r} |u - u(x)|^2$.
    \end{enumerate}
In particular, we have $N(x, r) \leq D(x, r) \leq N(x, 2r)$ and the monotonicity property for doubling index
\begin{equation}
     D(x, s) \leq D(x, r) \text{ for any } 0 < s \leq r/2. 
\end{equation}
\end{lemma}
\begin{proof}
    The first inequality follows directly from monotonicity of the frequency function for harmonic functions. To prove the second one, we assume $u(x) = 0$ and $r_1 = 1$ for simplicity. Then for $r>1$ we have
    \begin{equation}
    \begin{split}
     \quad \int_{B_{r}} u^2 &= \int_0^{r} ds \int_{\partial_{B_s}} u^2 \\
     &= r \int_0^{1} ds \int_{\partial_{B_{rs}}} u^2 \\
     & \leq r \int_0^{1} ds \big( r^{2N(x, rs) + n -1} \int_{\partial_{B_s}} u^2 \big) \\
     & \leq r^{2N(x, r) + n} \int_{B_1} u^2.
    \end{split}
    \end{equation}
For the lower bound, we have 
  \begin{equation}
    \begin{split}
     \quad \int_{B_{r}} u^2 &= \int_0^{r} ds \int_{\partial_{B_s}} u^2 \\
     &= r \int_0^{1} ds \int_{\partial_{B_{rs}}} u^2 \\
     & \geq r \int_0^{1} ds \big( r^{2N(x, s) + n -1} \int_{\partial_{B_s}} u^2 \big) \\
     & \leq r^{2N(x, 0) + n} \int_{B_1} u^2.
    \end{split}
    \end{equation}

The proof of (2) is now finished by scaling.

Finally we prove (3). We also assume $x = 0, r=1$ and $u(0) = 0$. Hence
\begin{equation}
    \fint_{B_1} u^2 = \int_0^1 r^{n-1} H(r) dr \leq \frac{H(1)}{n} = \frac{1}{n} \fint_{\partial B_1} u^2.
\end{equation}

On the other hand, for $s \leq 1$ we have
\begin{equation}
    \fint_{B_1} u^2 = \int_0^1 r^{n-1} H(r) dr \geq \int_0^1 r^{2N(1) + n-1} H(1) dr = \frac{1}{n + 2N(1)}H(1) \geq \frac{1}{n + 2D(0,1)}H(1).
\end{equation}

The proof is now finished.
\end{proof}

We recall more useful properties of doubling index for harmonic functions. First note that the the doubling index for harmonic functions can only be pinched when around integers. 

\begin{lemma}\label{l:Harm_pinch_near_Z}
Let $h: B_2 \to \RR$ be a harmonic function. For any $\epsilon \leq \epsilon_0(n)$, if $|D(0, 1) - D(0, 1/20)| \leq \epsilon$, then there exists some integer $d$ such that for any $s \in (1/10, 1/2)$ we have
\begin{equation}
    |D(0, s) - d| \leq 3\epsilon. 
\end{equation}
\end{lemma}

\begin{proof}
    This follows from the similar results for frequency. Indeed, if $|D(0, 1) - D(0, 1/20)| \leq \epsilon$, then by Lemma \ref{l:Harmonic_AlmostMonotone} we have
    \begin{equation}
        N(0, 1) - N(0, 1/10) \leq \epsilon.\end{equation}
    Then by Theorem 3.19 in \cite{NV}, there exists some integer $d$ such that $|N(s) - d| \leq 3\epsilon$ for any $s\in (1/10, 1)$. The proof is finished by Lemma \ref{l:Harmonic_AlmostMonotone} again and the monotonicity of the frequency function.
\end{proof}

The following Lemma says that the doubling index drops at a definite rate when away from integers.

\begin{lemma}\label{l:DI_Drop_Harmonic}
Let $h: B_2 \to \RR$ be a harmonic function. For any $\epsilon \leq \epsilon_0$, if $D(0, 1) \leq d - \epsilon$ for some integer $d$, then we have
\begin{equation}
    D(0, \frac{\epsilon}{2}) \leq d -1 +\epsilon.
\end{equation}
In particular, if $D(0, 1) \leq d + 1/2$, then
\begin{equation}
    D(0, \frac{\epsilon}{2}) \leq d +\epsilon.
\end{equation}

\end{lemma} 

\begin{proof}
According to  Lemma 3.16 in \cite{NV}, we have
\begin{equation}
    D(0, r_1) - D(0,r_2) \geq N(0,r_1) - N(0, 2r_2) \geq 2\epsilon(1- \epsilon) \log(\frac{r_1}{2r_2}).
\end{equation}

    Since $N(0,1) \leq D(0,1) \leq d - \epsilon$, by Lemma 3.16 in \cite{NV}, we have 
    \begin{equation}
        N(0, \frac{\epsilon}{1-\epsilon}) \leq d-1-\epsilon. 
    \end{equation}
     Hence we have $D(0, \epsilon/2) \leq N(0, \epsilon) \leq d-1 -\epsilon$.
\end{proof}

Next we recall the uniqueness of tangent map for harmonic functions. We say $u$ is \textbf{uniformly $(k,\eta, x)$-symmetric in $[r_1, r_2]$} if there exists a $k$-symmetric polynomial $P$ with $\fint_{\partial B_1} |P|^2 = 1$ such that for any $r \in [r_1, r_2]$
\begin{equation}\label{d:uniform_symm}
    \fint_{\partial B_1} |u_{x, r} - P|^2 \leq \eta.
\end{equation}
Note that compared to Definition \ref{d:quan_sym}, the approximated polynomial $P$ here does not depend on the choice of $r$.

\begin{proposition}\label{p:harm_unique}
There exists some $\epsilon_0 > 0$ such that the following is true. Let $u: B_1 \to \RR$ be harmonic with $|D(0, r_2) - D(0, r_1)| \leq \epsilon \leq \epsilon_0$ with $r_2 \leq r_1/20$. Then $u$ is uniformly $(0, 7\epsilon, x)$-symmetric in $[3r_2, r_1/3]$.
\end{proposition}

\begin{proof}
    The proof follows directly from Theorem 3.19 in \cite{NV} and Lemma \ref{l:Harmonic_AlmostMonotone}.
\end{proof}

\begin{remark}
Under the conditions above, the pinched doubling index is close to some integer $d$ by Lemma \ref{l:Harm_pinch_near_Z}. One can also see from the proof that the approximated polynomial can be taken as the normalization of the $d$-th order part of the Taylor expansion of $u$.
\end{remark}

\subsection{Doubling Index over balls}
Note that previously we define the doubling index using $L^2$ average over spheres. We can also define the doubling indices over balls for elliptic solutions as follows
\begin{definition}
    Let $u: B_2 \to \RR$ be a solution to \ref{e:elliptic_equ} \ref{e:assumption} with doubling assumption \ref{e:DI_bound_assumption}. For any $x \in B_{1}, r \leq \frac{1}{(1+  \lambda)^{1/2}}$ , we define the  doubling index over balls as
    \begin{equation}
        \Tilde{D}_x(r) \equiv \text{log}_4 \frac{\fint_{ B_{2}(0)} (u(x +rA_x(y)) - u(x))^2 dy}{\fint_{ B_1(0)} (u(x +rA_x(y)) - u(x))^2 dy}.
    \end{equation}
\end{definition}

The result in this section is that the doubling index over balls is uniformly bounded under the doubling assumption \ref{e:DI_bound_assumption}.

\begin{lemma}\label{l:Bound_on_General_DI_ball}
    Let $u$ be a solution to (\ref{e:elliptic_equ}) (\ref{e:assumption}) with doubling assumption \ref{e:DI_bound_assumption}. For any $x \in B_{1}$ and $r \leq \frac{1}{(1+  \lambda)^{1/2}}$ we have
    \begin{equation}
        \Tilde{D}_x(r) \leq C_0(n, \lambda)\Lambda.
    \end{equation}
\end{lemma}

\begin{proof}
     Since $(1+\lambda)^{-1} \delta^{ij} \leq a^{ij} \leq (1+\lambda) \delta^{ij}$, we have
    \begin{equation}
     \int_{B_{(1+\lambda)^{-1}r}(x)} |u - u(x) |^2 \leq C(n,\lambda)\int_{B_r(0)} (u(x +A_x(y)) - u(x))^2 \leq  C(n,\lambda)\int_{B_{(1+\lambda)r}(x)} |u - u(x) |^2
    \end{equation}

Now by doubling assumption, we have
\begin{align}
          \Tilde{D}_x(r) &= -\frac{n}{2} +\text{log}_4 \frac{\int_{B_{2}(0)} (u(x +rA_x(y)) - u(x))^2dy}{\int_{B_{1}(0)} (u(x +rA_x(y)) - u(x))^2dy} \\
          &= -\frac{n}{2} +\text{log}_4 \frac{\int_{B_{2r}(0)} (u(x +A_x(y)) - u(x))^2dy}{\int_{B_{r}(0)} (u(x +A_x(y)) - u(x))^2dy} \\
          &\leq -\frac{n}{2} + C(n,\lambda)+\text{log}_4 \frac{\int_{B_{2(1+\lambda)r}(x)} |u - u(x) |^2}{\int_{B_{(1+\lambda)^{-1}r}(x)} |u - u(x) |^2} 
    \leq C(n,\lambda)\Lambda.
\end{align}
This completes the proof.
\end{proof}

In the next section, we will prove that the doubling index (over spheres) is also uniformly bounded. 

\subsection{Almost Monotonicity of Doubling Index}\label{ss:Almost_mono}

In this section, we prove the almost monotonicity property of doubling index. First we prove an approximation Lemma by harmonic functions.

\begin{lemma}{(Harmonic Approximation)}\label{l:Harm_Approx}
Let $u$ be a solution to (\ref{e:elliptic_equ}) (\ref{e:assumption}) with doubling assumption \ref{e:DI_bound_assumption}. Let $\epsilon \in (0, 1/10)$. There exists $r_0 = C(n, \lambda, \alpha)^{-\Lambda} \epsilon^{(n+2)/(2\alpha)}$ such that for any $x\in B_{1}$ and $r \leq r_0$, there exists a harmonic function $h:B_3 \to \RR$ such that 
\begin{equation}
    h(0) = 0, \quad \fint_{\partial B_1} h^2 = 1, \quad \text{ and } \quad \sup_{B_{3}} |h - u_{x,r} | \leq \epsilon.
\end{equation}
\end{lemma}

\begin{proof}

For any $x\in B_1$ and $r$ small, define
\begin{equation}
    \Tilde{u}(y) = \frac{\Big( \fint_{\partial B_1} u_{x,r}(z)^2 dz\Big)^{1/2}}{\Big( \fint_{B_1} u_{x,r}(z)^2 dz\Big)^{1/2}} u_{x,r}(y) = \Big( \fint_{\partial B_1} \Tilde{u}(z)^2 dz\Big)^{1/2} u_{x,r}(y).
\end{equation}

Hence we have $\fint_{B_1} \Tilde{u}^2 = 1$. By Lemma \ref{l:Bound_on_General_DI_ball}, we have
\begin{equation}
    \fint_{B_5} \Tilde{u}^2 \leq C(n,\lambda)^\Lambda. 
\end{equation}

By elliptic estimates we have
\begin{equation}
    \sup_{B_{4}} (|\Tilde{u} + |\nabla \Tilde{u}| ) \leq C(n,\lambda, \alpha)^\Lambda.
\end{equation}

Let $\Tilde{h}$ be the solution to the following Dirichlet problem
\begin{equation}
\begin{cases}
  \Delta \Tilde{h} = 0  & \text{ in } B_{4} \\
  \Tilde{h} = \Tilde{u} & \text{ on } \partial B_{4}.
\end{cases}
\end{equation}

Consider $v = \Tilde{u} -\Tilde{h}$. Since $\Tilde{u}$ satisfies the equation \ref{e:Rescaled_u_equation}, we have
\begin{equation}
    \Delta v = \Delta \Tilde{u} = \partial_i\big( (\Tilde{a}^{ij} - \delta^{ij}) \partial_j \Tilde{u} \big) + \Tilde{b}^i \partial_i \Tilde{u}.
\end{equation}

Multiplying the equation by $v$ and taking integral over $B_{3}$ on both sides, then integration by parts gives
\begin{equation}
\begin{split}
    \int_{B_{3}} |\nabla v|^2 &\leq  \sup_{B_{3}} (| \Tilde{a}^{ij} - \delta^{ij}| |\nabla \Tilde{u}|) \int_{B_{3}} |\nabla v|  + \sup_{B_{3}}(|\Tilde{b}^i||\nabla \Tilde{u}| ) \int_{B_{3}}  |v|  \\
    &\leq C(n,\lambda,\alpha)^\Lambda  r^{\alpha} \big( (\int_{B_{3}} |\nabla v|^2)^{1/2} +  (\int_{B_{3}} v^2)^{1/2} \big).
    \end{split}
\end{equation}

Since $v = 0$ on $\partial B_{4}$, by Poincare inequality, we have
\begin{equation}
    \int_{B_{4}} v^2 \leq C(n,\lambda) \int_{B_{4}} |\nabla v|^2. 
\end{equation}

Combining these two equations, it is obvious that
\begin{equation}
    \int_{B_{4}} v^2  \leq C(n,\lambda, \alpha)^{\Lambda} r^{2\alpha}.
\end{equation}

Hence $\int_{B_{4}} \Tilde{h}^2 \leq 2 \int_{B_4} (v^2 + \Tilde{u}^2)  \leq C(n,\lambda,\alpha)^{\Lambda}$. By elliptic estimates, we have
\begin{equation}
    \sup_{B_3} |\nabla \Tilde{h}| \leq C(n,\lambda,\alpha)^{\Lambda}.
\end{equation}

In a conclusion we have
\begin{equation}
\begin{cases}
    &\sup_{B_3} |\nabla v| \leq \sup_{B_3}( |\nabla \Tilde{h}| + |\nabla \Tilde{u}|) \leq C(n,\lambda,\alpha)^{\Lambda}  \\
    & \int_{B_{4}} v^2  \leq C(n,\lambda, \alpha)^{\Lambda} r^{2\alpha} .
\end{cases}
\end{equation}

Consider the set $A = \{ y \in B_3 : |v|(y) \geq \epsilon/10\}$. If $A = \emptyset$, then we are done. Suppose there exist some $y \in A$. Then the gradient bound on $|\nabla v|$ implies that
\begin{equation}
    |v|(z) \geq \epsilon/20, \text{ for any } z \in B_s(y)\cap B_3 \text{ with } s \geq \frac{\epsilon}{20} C(n,\lambda, \alpha)^{-\Lambda}  .
\end{equation}

Therefore,
\begin{equation}
    \int_{B_{7/2}} v^2 \geq \int_{B_s(y)} v^2 \geq C(n,\lambda,\alpha)^{-\Lambda} \epsilon^{n+2}.
\end{equation}

If we choose $r \leq C(n,\lambda,\alpha)^{-\Lambda} \epsilon^{n+2/(2\alpha)}$, then we obtain a contradiction. Hence the set $A = \emptyset$. This implies that $\sup_{B_3} |\Tilde{u} - \Tilde{h}| \leq \epsilon/10$.

Since $\fint_{B_1} \Tilde{u}^2 = 1$, we have $\fint_{B_1} \Tilde{h}^2 \leq 2$ by triangle inequality. Then
\begin{equation}
    \fint_{\partial B_1} \Tilde{u}^2 \geq  \frac{1}{2}\fint_{\partial B_1} (\Tilde{h}^2 - 2|\Tilde{u} - \Tilde{h}|^2) \geq \frac{1}{2}\fint_{\partial B_1} \Tilde{h}^2 - \frac{\epsilon^2}{100}
\end{equation}
and by Lemma \ref{l:Harmonic_AlmostMonotone} we have
\begin{equation}
    \fint_{\partial B_1} \Tilde{h}^2 \geq  \fint_{B_1} \Tilde{h}^2 \geq \frac{1}{2} \fint_{B_1} (\Tilde{u}^2 
 - 2|\Tilde{u} - \Tilde{h}|^2) \geq \frac{1}{2} -\frac{\epsilon^2}{100}
\end{equation}

Hence
\begin{equation}
    \fint_{\partial B_1} \Tilde{u}^2 \geq \frac{1}{4} - \frac{\epsilon^2}{100} - \frac{\epsilon^2}{200} \geq  \frac{1}{5}. 
\end{equation}

Therefore, if we choose $h = \big( \fint_{\partial B_1} \Tilde{u}^2 \big)^{-1/2} \Tilde{h}$, then
\begin{equation}
    \sup_{B_3} | u_{x,r} - h| = \big( \fint_{\partial B_1} \Tilde{u}^2 \big)^{-1/2} \sup_{B_3} |\Tilde{u} - \Tilde{h}| \leq 5 \sup_{B_3} |\Tilde{u} - \Tilde{h}| \leq \epsilon/2.
\end{equation}

 We may further assume ${h}(0) = 0$ under appropriate adjustments of ${h}$ since ${u}_{x,r}(0) = 0$. By proper scaling of $h$, the condition $\fint_{\partial B_1} h^2 = 1$ is easily satisfied since $\fint_{\partial B_1} u_{x,r}^2 = 1$.
\end{proof}

As a corollary of the harmonic approximation Lemma above, we prove that the doubling index is bounded when $r$ is small. It is important that the bound depends linearly on the bound in doubling assumption \ref{e:DI_bound_assumption}.

\begin{lemma}\label{l:Bound_on_General_DI}
    Let $u$ be a solution to (\ref{e:elliptic_equ}) (\ref{e:assumption}) with doubling assumption \ref{e:DI_bound_assumption}. For any $x \in B_{1}$ and $r \leq C(n,\lambda,\alpha)^{-\Lambda}$ for some constant $C(n,\lambda,\alpha) > 0$, we have
    \begin{equation}
        D(x,r) \leq C(n, \lambda)\Lambda.
    \end{equation}
\end{lemma}

\begin{proof}
Choosing $\epsilon = c_1(n,\lambda)^{\Lambda}$, by previous Lemma \ref{l:Harm_Approx}, with $r \leq C(n,\lambda,\alpha)^{-\Lambda}$ we can find a harmonic function $h$ such that 
    $$h(0)=0, \quad \fint_{\partial B_1} h^2 =1 \quad \text{and } \sup_{B_3}|h - u_{x,r}| \leq c_1(n,\lambda)^{\Lambda},$$
where $c_1$ is a sufficiently small constant.

 By Lemma \ref{l:Harmonic_AlmostMonotone}, 
\begin{equation}
    \fint_{B_1}  u^2_{x,r} \leq 2 \fint_{B_1} h^2 + 2 \sup_{B_1}|h - u_{x,r}|^2 \leq 2 \fint_{\partial B_1} h^2 + 2 c_1(n,\lambda)^{\Lambda} \leq 5.
\end{equation}

By Lemma \ref{l:Bound_on_General_DI_ball}, we have 
\begin{equation}
    \fint_{B_{3}}  u^2_{x,r} \leq 4^{2 C(n,\lambda)\Lambda} \fint_{B_1}  u^2_{x,r} \leq 4^{3C(n,\lambda)\Lambda}.
\end{equation}
and thus
\begin{equation}
     \fint_{B_3}  h^2 \leq 2 \fint_{B_3} u^2_{x,r} + 2 \sup_{B_3}|h - u_{x,r}|^2 \leq  4^{4C(n,\lambda)\Lambda}. 
\end{equation}

Note that by Lemma \ref{l:Harmonic_AlmostMonotone}, 
\begin{equation}
    \fint_{B_3}  h^2 \geq  c(n) \int_2^3 \fint_{\partial B_r} h^2 dr \geq c(n) \fint_{\partial B_2} h^2.
\end{equation}

Finally we have
\begin{equation}
    \fint_{\partial B_{2}}  u^2_{x,r} \leq 2 \fint_{\partial B_2} h^2 + 2\sup_{B_3}|h - u_{x,r}|^2 \leq 4^{6 C(n,\lambda)}.
\end{equation}

Therefore,
\begin{equation}
   D^u(x,r) = D^{u_{x,r}}(0,1)=\log_4 \frac{\fint_{\partial B_{2}} u_{x,r}^2}{\fint_{\partial B_{1}} u_{x,r}^2} \leq 6C(n,\lambda)\Lambda.
\end{equation}
\end{proof}

Another corollary is that the  doubling index is close to the doubling index of the approximated harmonic function if $r$ is chosen appropriately. This is crucial for us to construct the almost monotonicity property of doubling index for elliptic solutions.

\begin{lemma}\label{l:DI_close_away_from_0}
    Let $u$ be a solution to (\ref{e:elliptic_equ}) (\ref{e:assumption}) with doubling assumption \ref{e:DI_bound_assumption}. Let $\epsilon \in (0, 1/10]$. There exists some $r_0 =\epsilon^{C(n,\lambda,\alpha)\Lambda}$ such that for any $x \in B_{1}$ and $r \leq r_0$, there exists an approximated harmonic function $h$ to $u_{x,r}$ with 
\begin{equation}
    |D^h(0, s) - D^{u_{x,r}}(0, s)| \leq \epsilon,
\end{equation}
for $ \epsilon^2 \leq s \leq 1$. 
\end{lemma}

\begin{proof}

 Denote $u_{x,r}$ by $\Tilde{u}$ with $r \leq r_0 =\epsilon^{C(n,\lambda,\alpha)\Lambda}$. Let $C_0(n,\lambda)$ be the doubling index bound in Lemma \ref{l:Bound_on_General_DI_ball} and \ref{l:Bound_on_General_DI}, i.e. $\Tilde{D}_x(r) \leq C_0$ and $D(x,r) \leq C_0$ for $r\leq r_0$. 
 
 By Lemma \ref{l:Harm_Approx}, there exists a harmonic function $h$ such that 
\begin{equation}
    \sup_{B_3} |h- \Tilde{u}|^2 \leq  \epsilon^{10 C_0\Lambda} .
\end{equation}

According to Young's inequality, for any $t$ we have
\begin{equation}
    \fint_{\partial B_t} h^2 - \Tilde{u}^2 = \fint_{\partial B_t} h^2 - |h - \tilde{u} - h|^2 = \fint_{\partial B_t} 2 h (h - \tilde{u}) - |h - \tilde{u}|^2   
    \leq \frac{1}{2}\fint_{\partial B_t} h^2 + |h - \tilde{u}|^2.
\end{equation}

If $\epsilon \leq 1/10$, by Lemma \ref{l:Harmonic_AlmostMonotone}, we have for any $t \in [1, 2]$, 
\begin{equation}
    \fint_{\partial B_t} \Tilde{u}^2 \geq \frac{1}{2} \fint_{\partial B_t} h^2 - \sup_{B_2} |h-\Tilde{u}|^2 
    \geq \frac{1}{2} - \epsilon^{10C_0\Lambda} \geq \epsilon^{C_0\Lambda}.
\end{equation}

Now for any $s < 1$, choose $i$ such that $1 \leq 2^{i+1} s < 2$. By the doubling index bound in Lemma \ref{l:Bound_on_General_DI}, if $s \geq \epsilon^2$, then 
\begin{equation}
     \fint_{\partial B_{2s}} \Tilde{u}^2 \geq 4^{-C_0\Lambda i}  \fint_{\partial B_{2^{i+1}s}} \Tilde{u}^2 \geq  \epsilon^{4 C_0\Lambda} \fint_{\partial B_{2^{i+1}s}} \Tilde{u}^2 \geq \epsilon^{5 C_0\Lambda}.
\end{equation}

Therefore we can write
\begin{equation}
    \fint_{\partial B_s} |h - \Tilde{u}|^2 \leq \sup_{B_2} |h - \Tilde{u}|^2 \leq \epsilon^{10C_0\Lambda} \leq \epsilon^{5C_0\Lambda}  \fint_{\partial B_s} \Tilde{u}^2.
\end{equation}

Therefore by Young's inequality, 
\begin{equation}
    \fint_{\partial B_s} |h + \Tilde{u}|^2 \leq 2 \fint_{\partial B_s} (|h - \Tilde{u}|^2 + 4\Tilde{u}^2 ) \leq 10\fint_{\partial B_s} \Tilde{u}^2.
\end{equation}

Hence we have
\begin{equation}
    \fint_{\partial B_s} |h^2 - \Tilde{u}^2| = \fint_{\partial B_s} |h - \Tilde{u}| | h + \Tilde{u}| \leq \frac{1}{2} \fint_{\partial B_s} (\frac{\epsilon}{100} |h + \Tilde{u}|^2 + \frac{100}{\epsilon} |h- \Tilde{u}|^2) \leq \frac{\epsilon}{10}  \fint_{\partial B_s} \Tilde{u}^2. 
\end{equation}

This proves that for any $s \in [\epsilon^2, 1]$,
\begin{equation}
    |\log_4\frac{\fint_{\partial B_{2s}} h^2}{\fint_{\partial B_{s}} h^2} - \log_4\frac{\fint_{\partial B_{2s}} \Tilde{u}^2}{\fint_{\partial B_{s}} \Tilde{u}^2}| \leq \log_4 \frac{1+\epsilon/10}{1-\epsilon/10} \leq \epsilon.
\end{equation}
\end{proof}

Now we are in the position to prove the almost monotone theorem for doubling index $D(x,r)$.

\begin{theorem}\label{t:Almost_monotone}
Let $u: B_2 \to \RR$ be  a solution to (\ref{e:elliptic_equ}) with (\ref{e:assumption}) with doubling assumption \ref{e:DI_bound_assumption}. For any $\epsilon \in (0, 1/10]$, there exists some $r_0 =\epsilon^{C(n,\lambda,\alpha)\Lambda}$ such that the following almost monotonicity formula holds for the doubling index:
\begin{equation}
    D(x,s) \leq D(x,r) + \epsilon,
\end{equation}
for any $x\in B_{1}$ and $ 0 < 2s \leq r \leq r_0$. In particular, we have $D(x,r) \geq 1 - \epsilon$ for any $0 < r \leq r_0$.
\end{theorem}

\begin{proof}
Choose $\epsilon$ to be small and fix $x \in B_{1}$. Then by Lemma \ref{l:DI_close_away_from_0}, there exists some $r_0$ such that for any $r \leq r_0$, there exists some harmonic function $h_{x,r}$ such that for any $s \geq \epsilon^2$,
\begin{equation}
    |D^{h_{x,r}}(0,s) - D^{u_{x, r}}(0, s)| \leq \epsilon/40.
\end{equation}

According to remark \ref{r:Gen_DI}, we have $D^u(x,r) = D^{u_{x,r}}(0,1)$. Choose $d$ such that 
\begin{equation}
    |D^{u_{x,r}}(0,1) - d| = \min_{k \in \mathbf{Z}} |D^{u_{x,r}}(0,1) - k| \leq 1/2.
\end{equation} 

In the following we write $D(x, r) = D^u(x, r)$. 

\textbf{Claim: } Suppose $D(x,r) \leq d + \epsilon/10$. Then $D(x,s) \leq d + \epsilon/2$ for any $0 \leq s \leq r/2$.

\begin{proof}
    We define
    \begin{equation}
    t_0 \equiv \inf \{ t\geq 0 : D(x,s) \leq d + \epsilon/4 \text{ for any } s\in [t, r/2] .\}
    \end{equation}
    By the doubling index convergence, for any $s \in [\epsilon^2 r, r/2]$,
    \begin{equation}
        D(x,s) \leq D^{h_{x,r}}(0, s/r) + \epsilon/20 \leq D^{h_{x,r}}(0, 1) +\epsilon/20 \leq D(x,r) +\epsilon/10.
    \end{equation}

Hence $t_0$ exists and $t_0 \leq \epsilon^2 r$. If $t_0 =0$, then the claim holds trivially. We suppose $t_0 > 0$. Now with the same $\epsilon$, we apply the harmonic approximation lemma \ref{l:Harm_Approx} and doubling convergence Lemma \ref{l:DI_close_away_from_0} again at the scale $t_0 < r$. For any $s \in [\epsilon^2, 1]$ we have
    \begin{equation}
    |D^{h_{x,t_0}}(0,s) - D^{u_{x, t_0}}(0, s)| \leq \epsilon/40.
    \end{equation}

By triangle inequality, for any $\epsilon^2 t_0 \leq s \leq t_0$, we have
\begin{equation}
\begin{split}
    D(x,s) &= D^{u_{x, t_0}}(0, s/t_0) \leq D^{h_{x, t_0}}(0, s/t_0) + \epsilon/40 \\
    &\leq D^{u_{x, t_0}}(0, 2) + \epsilon/20 \\
    &= D(x,2t_0) + \epsilon/20 \leq d+ \epsilon/4 + \epsilon/20 \\
    &\leq d+ 3 \epsilon/ 10.
\end{split}
\end{equation}

Hence we conclude that $D(x,s) \leq d+3\epsilon/10$ for any $s \in [\epsilon^2 t_0, r]$.

Moreover, since $D^{h_{x,t_0}}(0,1) \leq D^{u_{x, t_0}}(0, 1) + \epsilon/40 \leq d + 13\epsilon/40$, by Lemma \ref{l:DI_Drop_Harmonic}, we have 
    \begin{equation}
        D^{h_{x,t_0}}(0, 2\epsilon^2) \leq d + 4\epsilon^2.
    \end{equation}

Hence $D(x,\epsilon^2 t_0) = D^{u_{x, t_0}}(0, \epsilon^2) \leq D^{h_{x,t_0}}(0, 2\epsilon^2) + \epsilon/40 \leq d + \epsilon/10$. 

Now by induction for any $i \geq 0$ we can define 
\begin{equation}
    t_{i+1} \equiv \inf \{ t\geq 0 : D(x,s) \leq d + \epsilon/4 \text{ for any } s\in [t, \epsilon^2 t_i] .\}
    \end{equation}
and by the same arguments as above we can prove that 
\begin{enumerate}
    \item $t_{i+1} \leq \epsilon^2 t_i$.
    \item $D(x,s) \leq d+3\epsilon/10$ for any $s \in [\epsilon^2 t_{i+1}, r]$.
    \item $D(x,\epsilon^2 t_{i+1}) \leq d + \epsilon/10$.
\end{enumerate}

The proof of the claim is finished by induction.
\end{proof}

Next we prove the main theorem. We consider the following three cases.

\textbf{Case 1:} If $d - 1/2 \leq D(x,r) \leq d - \epsilon/10$, then we apply Lemma \ref{l:DI_Drop_Harmonic} to the harmonic approximate function $h_{x,r}$ to obtain that
\begin{equation}
    D(x,\epsilon r /50)=D^{u_{x,r}}(0,\epsilon/50) \le  D^{h_{x,r}}(0, \epsilon/50)+\epsilon/40 \leq d-1+\epsilon/20.
\end{equation}
By triangle inequality, we have
\begin{equation}
    D(x,s) \leq D(x,r) + \epsilon/2 \text{ for any  } s \in [\epsilon r/100, r/2],
\end{equation}
and meanwhile, by the above claim, since $D(x,\epsilon r/50) \leq d -1 +\epsilon/20$ we have
\begin{equation}
    D(x,s) \leq d-1 +\epsilon/2 \leq  D(x,r) \text{ for any } s\leq \epsilon r/100.
\end{equation}

\textbf{Case 2:} If $d - \epsilon/10 \leq D(x,r) \leq d+ \epsilon /10$, then according to the claim, for any $s \leq r/2$ we have
\begin{equation}
    D(x,s) \leq d + \epsilon/2 \leq D(x,r) + \epsilon.
\end{equation}

\textbf{Case 3:} If $d + \epsilon/10 \leq D(x,r) \leq d + 1/2$, then we can apply the Lemma \ref{l:DI_Drop_Harmonic} and the doubling convergence Lemmma \ref{l:DI_close_away_from_0} to obtain that
\begin{equation}
    D(x,s) \leq D(x,r) + \epsilon/2 \text{ for any  } s \in [\epsilon r/100, r/2],
    \end{equation}
and 
\begin{equation}
     D^{h_{x,r}}(0, \epsilon/50) \leq d+\epsilon/20,
\end{equation}
and 
\begin{equation}
    | D^{h_{x,r}}(0, \epsilon/50) -D(x,\epsilon r/50)|\le \epsilon/40.
\end{equation}
The proof is now finished by the claim.
\end{proof}

As a corollary, we prove that there are only finitely many non-pinching scales for doubling index. Usually this type of property is proved using monotonicity. However, we observe that for this to hold even the almost monotonicity suffices.  

\begin{corollary}\label{c:finite_non_pinch}
Let $u: B_2 \to \RR$ be  a solution to (\ref{e:elliptic_equ}) with (\ref{e:assumption}) with doubling assumption \ref{e:DI_bound_assumption}. For any $\epsilon \in (0, 1/10]$ and $x \in B_1$, let $I \equiv \{i \in \mathbb{N} : |D(x, 4^{-i}) - D(x, 4^{-i-1})| \geq \epsilon \}$ to denote the set of non-pinching scales. Then the number of the elements in $I$ is bounded by some constant $C(n,\lambda, \alpha)\epsilon^{-1}\Lambda$.
\end{corollary}

\begin{proof}
    When $i \geq i_0 \equiv C(n,\lambda,\alpha)\epsilon^{-1}\Lambda$, we can apply the Almost Monotonicity Theorem \ref{t:Almost_monotone} to obtain that $D(x, 4^{-i-1}) \leq D(x, 4^{-i}) + \epsilon/2$. If moreover $i \in I$, then the only possibility is that
    \begin{equation}\label{e:finite_non_pinch}
        D(x, 4^{-i-1}) \leq D(x, 4^{-i}) - \epsilon.  
    \end{equation}

Write $I_{i \geq i_0} = \{i_0, i_1, ..., i_k, ...\}$ in an ascending order. For each $k$ we have
\begin{equation}
    D(x, 4^{-i_{k}}) \geq D(x, 4^{-i_{k} - 1}) + \epsilon \geq D(x, 4^{-i_{k+1}}) - \epsilon/2 + \epsilon = D(x, 4^{-i_{k+1}}) + \epsilon/2,
\end{equation}
where the first inequality follows \ref{e:finite_non_pinch} and the second inequality follows from Almost Monotonicity Theorem. 

We can apply Lemma \ref{l:Bound_on_General_DI} to get
\begin{equation}
    C(n, \lambda) \Lambda \geq D(x, 4^{-i_{0}}) \geq D(x, 4^{-i_{1} }) + \epsilon/2 \geq D(x, 4^{-i_{K} }) + K\epsilon/2 \geq K\epsilon/2.
\end{equation}
This implies that $K \leq C(n,\lambda)\epsilon^{-1}\Lambda$ and thus that $|I| \leq i_0 + K \leq C(n,\lambda, \alpha)\epsilon^{-1}\Lambda.$
\end{proof}

We will prove more properties for doubling index in the next section.

\section{Uniqueness of Tangent maps}

In this chapter, we will prove a quantitative version of uniqueness of tangent maps for elliptic solutions to \ref{e:elliptic_equ}, \ref{e:assumption} with doubling assumption \ref{e:DI_bound_assumption}. 

Note that in \cite{Hansingular}, he proved the uniqueness of tangent map for H\"older coefficients of equations in non-divergence form.  Following the argument by Han \cite{Hansingular}, Naber and Valtorta \cite{NV} proved a quantitative uniqueness result of tangent map for Lipschitz coefficients of equation in divergence form which is also non-divergence with Lipschitz coefficients. In theses cases, an a-priori $W^{2,p}$-theory is valid. However, in our case, the coefficient is only H\"older and there is no $W^{2,p}$-theory. 
To overcome this difficulty, we use integration by  parts basing on the growth estimates for the gradient.

\subsection{Homogeneous Harmonic Polynomials}

In this section, we collect some useful facts about homogeneous harmonic polynomials. We first define the inner product of homogeneous polynomials $P_1$ and $P_2$ as 
\begin{equation}
    \langle P_1, P_2 \rangle \equiv \fint_{\partial B_1} P_1 P_2.
\end{equation}

And the norm $||P||$ is induced from the inner product. By simple calculations we have the following equality for a homogeneous harmonic polynomial of degree $d$ 
\begin{equation}
    ||P||^2 = (n + 2d) \fint_{B_1} P^2. 
\end{equation}

We will collect some well-known results for homogeneous harmonic polynomials for later use. First we define $\cP_d$ as the set of all homogeneous harmonic polynomials of degree $d$. The first Lemma is a relation between the $L^2$-norm of $P$ and $\nabla P$ for $P \in \cP_d$ as in Lemma 5.13 \cite{ABR}. 

\begin{lemma}\label{l:hhp_gradient_norm}
    Let $P_1$, $P_2 \in \cP_d$. Then
    \begin{equation}
        \langle \nabla P_1,  \nabla P_2 \rangle = d(2d + n -2) \langle P_1, P_2 \rangle. 
    \end{equation}
\end{lemma}

In the next Lemma, we can control the $L^{\infty}$-norm using the $L^2$-norm and the degree of a homogeneous harmonic polynomial. See Lemma 3.3 \cite{NV}.

\begin{lemma}\label{l:hhp_C0_bound}
    Let $P \in \cP_d$. Then
    \begin{equation}
        ||P||_{C^0(B_1)} \leq C(n) d^{\frac{n}{2} -1} ||P||.
    \end{equation}
\end{lemma}
\begin{proof}
Without loss of generality, assume $||P||=1$. We can get 
\begin{align}
\int_{B_{1+\epsilon}}|P|^2=\int_0^{1+\epsilon}\int_{\partial B_r}|P|^2 =\frac{1}{2d+n}(1+\epsilon)^{2d+n}.
\end{align}
By elliptic estimate, 
\begin{align}
\sup_{B_1}|P|\le C(n)\epsilon^{-n/2}\left(\int_{B_{1+\epsilon}}|P|^2\right)^{1/2}\le \frac{C(n)\epsilon^{-n/2}}{2d+n}(1+\epsilon)^{2d+n}.
    \end{align}
Letting $\epsilon=1/d$, we finish the proof.
\end{proof}

As we have to deal with the symmetry of polynomials quite often, we define $\cP_d(\cancel{x})$ to be the subspace of $\cP_d$ of polynomials invariant along the direction of $x$, i.e. $\nabla P \cdot x \equiv 0$. This uniquely determines its orthogonal complement subspace $\cP_d(\cancel{x})^{\perp}$ as 
\begin{equation}
    \cP_d = \cP_d(\cancel{x}) \oplus \cP_d(\cancel{x})^{\perp}.
\end{equation}

The following property is important for the cone-splitting Lemmas. The proof can be found in Proposition 3.9 in \cite{NV}.

\begin{lemma}\label{l:Invariant_x_1}
    Let $P \in \cP_d(\cancel{x_1})^{\perp}$, then
    \begin{equation}
        ||P|| \leq ||\partial_1 P||.
    \end{equation}
\end{lemma}

\subsection{Growth Estimate for Gradient}

In this section we prove the growth estimate for gradient under the assumption of lower bound of doubling index. We are going to apply the lemma to $\Tilde{u} = u_{x, r_1}$ on the interval $[r_2, r_1]$ where the doubling index is pinched.

\begin{lemma}\label{l:Gradient_Growth}
    Let $u$ be a solution to \ref{e:elliptic_equ} \ref{e:assumption} with doubling assumption \ref{e:DI_bound_assumption}. Let $10r_2 \leq r_1 \leq r_0 = C(n,\lambda,\alpha)^{-\Lambda}$. Assume $D(x, s) \geq \gamma$ for any $r_2 \leq s \leq r_1 $. Let $\Tilde{u} = u_{x, r_1}$ and $r_2' = r_2/r_1$. Then we have,
\begin{enumerate}

    \item \begin{equation}
    \fint_{B_t} \Tilde{u}^2 \leq \left\{
    \begin{array}{lr}
        C(n,\lambda)^\Lambda r_2'^{2\gamma - 9/5} t^{9/5}, &\text{if } t \leq r'_2.\\
        C(n,\lambda)^\Lambda t^{2\gamma}, &\text{if } t \in [r'_2, 2].
    \end{array} 
    \right.
    \end{equation}
    
    \item \begin{equation}
    \sup_{B_t} |\Tilde{u}|  \leq 
    \left\{
    \begin{array}{lr}
        C(n,\lambda,\alpha)^\Lambda r_2'^{\gamma - 9/10} t^{9/10}, &\text{if } t \leq r'_2.\\
        C(n,\lambda,\alpha)^\Lambda t^{\gamma}, &\text{if } t \in [r'_2, 1].
    \end{array} 
    \right.
\end{equation}

    \item \begin{equation}
    \sup_{B_t}  |\nabla \Tilde{u}| \leq 
    \left\{
    \begin{array}{lr}
        C(n,\lambda,\alpha)^\Lambda r_2'^{\gamma - 1} , &\text{if } t \leq r'_2.\\
        C(n,\lambda,\alpha)^\Lambda t^{\gamma - 1}, &\text{if } t \in [r'_2, 1].
    \end{array} 
    \right.
\end{equation}
\end{enumerate}
\end{lemma}

\begin{proof}
By harmonic approximation Lemma \ref{l:Harm_Approx} and Lemma \ref{l:Harmonic_AlmostMonotone}, for any $t \in [1, 2]$ we have
\begin{equation}
    \fint_{\partial B_{t}} \Tilde{u}^2 \leq C(n,\lambda)^\Lambda.
\end{equation}

By the lower bound on doubling index, for any $s \in [r'_2, 1]$,
\begin{equation}
    \fint_{\partial B_s} \Tilde{u}^2 \leq 4^{-\gamma i} \fint_{\partial B_{2^is}} \Tilde{u}^2 \leq C(n,\lambda)^\Lambda s^{2\gamma},
\end{equation}
where we choose $i$ satisfying $1 \leq 2^i s < 2$. 

Next we choose $\epsilon = 1/10$ and apply the Almost monotonicity theorem \ref{t:Almost_monotone}, we obtain that $D(x,r) \geq 9/10$ for any $r \leq r_0 = C(n,\lambda,\alpha)^{-\Lambda}.$ Then for any $s \leq r'_2$, we have
\begin{equation}
    \fint_{\partial B_s} \Tilde{u}^2 \leq 4^{-9i/10} \fint_{\partial B_{2^is}} \Tilde{u}^2 \leq
    C(n,\lambda)^\Lambda r_2'^{2\gamma - 9/5} s^{9/5}.
\end{equation}
where we choose $i$ satisfying $r'_2 \leq 2^i s < 2r'_2$.

Then if $t \leq r_2'$, we have
\begin{equation}
    \int_{B_t} \Tilde{u}^2 = \int_0^t \int_{\partial B_s} \Tilde{u}^2 ds \leq C(n,\lambda)^\Lambda r_2'^{2\gamma - 9/5}\int_0^t s^{n -1 + 9/5} ds \leq C(n,\lambda)^\Lambda r_2'^{2\gamma - 9/5} t^{n + 9/5}.
\end{equation}

If $t \in [r_2', 2]$, then
\begin{equation}
    \int_{B_t} \Tilde{u}^2 = \int_{B_{r_2'}} \Tilde{u}^2 + \int_{r_2'}^t \int_{\partial B_s} \Tilde{u}^2 ds \leq C(n,\lambda)^\Lambda r_2'^{2\gamma +n}  + C(n,\lambda)^\Lambda \int_{r_2'}^t s^{n- 1 + 2\gamma} ds \leq C(n,\lambda)^\Lambda t^{n+2\gamma}.
\end{equation}

This finishes the proof of the first inequality. The remaining two inequalities are proved by elliptic estimates (see \cite{GT}). 
\end{proof}

\subsection{Green's Function}

Following \cite{Hansingular} and \cite{NV}, we analyze the growth of Green's kernel of Laplacian in $\RR^n$ for $n\geq 3$. What is different here is that we need to get estimate of the gradient of the expansion terms. Similar results hold for $n=2$. 

Let $\Gamma(x, y) = c(n)|x - y|^{2-n}$ be the fundamental solution to Laplacian equation in $\RR^n$ for $n\geq 3$. For each $y \neq 0$, we have the expansion for $\Gamma_y(x) \equiv \Gamma(x,y)$ at $x=0$ in $B_{|y|}$
    \begin{equation}
        \Gamma_y(x) = \sum_k \Gamma_k(y) P_{y,k}(x),
    \end{equation}
    where $P_{y,k}(x)$ are homogeneous harmonic polynomials of degree $k$ with normalization $\fint_{\partial B_1} P_{y,k}^2 = 1$. Define 
    \begin{equation}
        R_{y, d}(x) \equiv \Gamma(x, y) - \sum_{k=0}^d \Gamma_k(y) P_{y,k}(x).
    \end{equation}

Note that $R_{y, d}$ is a harmonic function with vanishing order at the origin no less than $d+1$.

\begin{lemma}\label{l:Green_Kernel_growth}
    Assume as above. Then for any $y \neq 0$,  
    \begin{equation}
        |\Gamma_k(y)| \leq c(n) \Big( \frac{4}{3} \Big)^k |y|^{2-n-k} \quad \text { and } \quad   |\nabla_y \big( \Gamma_k(y) P_{y,k}(x) \big) | \leq c(n)\big( \frac{4}{3} \Big)^k k^{\frac{n}{2} - 1} |y|^{1-n-k} | |x|^k
    \end{equation}
For any $|x| \leq |y|/2$ we have
    \begin{equation}
    |R_{y,d}(x)|\leq c(n) 2^{d+1} |x|^{d+1} |y|^{1-n-d} \quad \text{ and } \quad
    |\nabla_y R_{y,d}(x)| \leq c(n) 2^{d+1} |x|^{d+1} |y|^{-n-d}.
\end{equation}
\end{lemma}

\begin{proof}
     By orthogonality of $P_{y,k}$, for any $r < |y|$ and any $k$,
    \begin{equation}
        \Big(\Gamma_k(y) r^k \Big)^2 \leq c(n)\fint_{\partial B_r} \Gamma_y(x)^2 dx = c(n)\fint_{\partial B_r} |x - y|^{4-2n} dx.
    \end{equation}
    
    Choose $r = 3|y| / 4$, then we have
    \begin{equation}
        |\Gamma_k(y)| \leq c(n) \Big( \frac{4}{3} \Big)^k |y|^{2-n-k}. 
    \end{equation}

Now for $y \neq 0$, we consider the expansion of $\partial_{y_i} \Gamma(x,y)$ at $x=0$
\begin{equation}
    \partial_{y_i} \Gamma(x,y) = \sum_k \Tilde{\Gamma}^i_k(y) \Tilde{P}^i_{y,k}(x).
\end{equation}
where each $\Tilde{P}^i_{y,k}(x)$ is a normalized homogeneous harmonic polynomial of degree $k$. 

Since $\partial_{y_i} \Gamma(x,y) = c(n)|x-y|^{-n}(y_i - x_i)$, by the same arguments as above, we have
\begin{equation}
    |\Tilde{\Gamma}^i_k(y) | \leq c(n)\big( \frac{4}{3} \Big)^k |y|^{1-n-k}.
\end{equation}

Since both of $P_{y,k}(x)$ and $\Tilde{P}^i_{y,k}(x)$ are homogeneous harmonic polynomials of degree $k$, we have 
\begin{equation*}
    \partial_{y_i} \big( \Gamma_k(y) P_{y,k}(x) \big) = \Tilde{\Gamma}^i_k(y) \Tilde{P}^i_{y,k}(x),
\end{equation*}
and thus
\begin{equation*}
    |\partial_{y_i} \big( \Gamma_k(y) P_{y,k}(x) \big) | \leq c(n)\big( \frac{4}{3}  \Big)^k k^{\frac{n}{2} - 1} |y|^{1-n-k} ~ |x|^k, 
\end{equation*}
where we have used the sharp upper bound for $C^0$-norm of homogeneous harmonic polynomials of degree $k$ as Lemma \ref{l:hhp_C0_bound}
\begin{equation*}
    ||P_k||_{C^0(B_r)} 
    \leq C(n) k^{\frac{n}{2} - 1} \big( \fint_{\partial B_r} P_k^2\big)^{1/2}.
\end{equation*}

This finishes the proof of the first part. Next we prove the bounds for $R_{y,d}$. Similarly, by orthogonality of $P_k$, we have
\begin{equation}
    \fint_{\partial B_{3|y|/4}} R_{y,d}(x)^2 \leq  \fint_{\partial B_{3|y|/4}} \Gamma(x,y)^2 \leq \frac{c(n)}{|y|^{2n -4}}.
\end{equation}
    
    Therefore, for $2|x| \leq |y|$ we have
    \begin{equation}
\begin{split}
     |R_{y,d}(x)| & \leq c(n) \big( \fint_{\partial B_{3|x|/2}} R_{y,d}(z)^2 \big)^{1/2}\\
     & \leq c(n) |2x|^{d+1} |y|^{-(d+1)} ( \fint_{\partial B_{3|y|/4}} R_{y,d}(z)^2 \big)^{1/2}\\
     & \leq c(n) 2^{d+1} |x|^{d+1} |y|^{1-n-d},
\end{split}
    \end{equation}
where the first inequality comes from elliptic estimates and the second comes from the fact that the frequency is no less than $d+1$. 

Note that 
\begin{equation}
    \partial_{y_i} R_{y,d}(x) = \partial_{y_i} \Gamma(x,y) - \sum_{k=0}^d \partial_{y_i} \big( \Gamma_k(y) P_{y,k}(x) \big) = \sum_{k=d+1}^{\infty} \Tilde{\Gamma}^i_k(y) \Tilde{P}^i_{y,k}(x).
\end{equation}
By the orthogonality of $\Tilde{P}^i_{y,k}(x)$ we have
\begin{equation}
    \fint_{\partial B_{3|y|/4}} |\partial_{y_i} R_{y,d}(x)|^2 \leq  \fint_{\partial B_{3|y|/4}} |\partial_{y_i} \Gamma(x,y)|^2 \leq \frac{c(n)}{|y|^{2n -2}}.
\end{equation}
and for any $2|x| \leq |y|$,
\begin{equation}
    |\partial_{y_i} R_{y,d}(x)| \leq c(n) \big( \fint_{\partial B_{3|x|/2}} \partial_{y_i} R_{y,d}(x)^2 \big)^{1/2} \leq c(n) |2x|^{d+1} |y|^{-(d+1)} ( \fint_{\partial B_{3|y|/4}} \partial_{y_i} R_{y,d}(x)^2 \big)^{1/2} \leq c(n) 2^{d+1} |x|^{d+1} |y|^{-n-d}.
\end{equation}
This finishes the proof. 
\end{proof}

\subsection{Quantitative Uniqueness of Tangent Maps}

In this section, we prove the quantitative version of uniqueness of tangent maps for solutions to \ref{e:elliptic_equ} \ref{e:assumption} and \ref{e:DI_bound_assumption}. First we prove an important harmonic approximation theorem, which can be viewed as an improved version of \ref{l:Harm_Approx} when the doubling index is pinched. Also see \cite{Hansingular} \cite{NV}.

\begin{proposition}\label{p:harmonic_appro}
Let $u: B_2 \to \RR$ be a solution to (\ref{e:elliptic_equ}) (\ref{e:assumption}) with doubling assumption \ref{e:DI_bound_assumption}. Let $\epsilon \leq 1/10$ and $x\in B_1$. Then there exists some $r_0 = C(n, \lambda, \alpha)^{-\Lambda} \epsilon^{1/\alpha}$ such that if $10r_2 \leq r_1\le r_0$ and $|D(x,s) - D(x,t)| \leq  \alpha$ for any $s, t \in [r_2, r_1]$. there exists some harmonic function $h$ in $B_1$ such that 
    
\begin{equation}\label{e:Harmonic_close}
    |h(y) - u_{x, r_1}(y) | \leq
    \left\{
    \begin{array}{lr}
        \epsilon \Big( \fint_{\partial B_{|y|}} u^2_{x,r_1} \Big)^{1/2}, &\text{if } |y| \in [r_2/r_1, 1/2].\\
        \epsilon \Big( \fint_{\partial B_{r_2/r_1}} u^2_{x,r_1} \Big)^{1/2}, &\text{if } |y| \leq r_2/r_1.
    \end{array} 
    \right.
  \end{equation}

In particular, we have the doubling index closeness
\begin{equation}
    |D^{h}(0, r) - D^{u_{x, r_1}}(0, r) | \leq 10 \epsilon \text{ for any } r \in [r_2/r_1, 1/2].   
\end{equation}
\end{proposition}
\begin{proof}

    In the following we prove the theorem when $n \geq 3$. The case $n=2$ is similar. Consider the fundamental solution to Laplacian equation $\Gamma(x,y) = c(n)|x-y|^{2-n}$. 
    
    By convolving the coefficients of the equation \ref{e:elliptic_equ} with a mollifier, we consider the elliptic equation with smooth coefficients $a_{\delta}^{ij}$ and $b^i_{\delta}$ such that $a_{\delta}^{ij} \to a^{ij}$ and $b^i_{\delta} \to b^i$ uniformly as $\delta \to 0$. We consider the solution $u_{\delta}$ to such an equation. In this way $\Delta u_{\delta}$ is well-defined. The proof will be finished once we can prove the result for $u_{\delta}$ provided only the H\"older norm of $a^{ij}$ and $L^{\infty}$-norm of $b^i$ are assumed. Hence in the following we will focus on the solutions $u_{\delta}$ and for simplicity we will omit the subscript $\delta$. 

{Denote $\Tilde{u} = u_{x, r_1}$ and note that $\Tilde{u}$ satisfies the equation}
\begin{equation}
     \Delta \Tilde{u}(y) = \partial_i (a^{ij}(0) - a^{ij}(y) \partial_j \Tilde{u}) - b_i(y) \partial_i \Tilde{u}. 
    \end{equation}
where the coefficients satisfy
\begin{equation}
   a^{ij}(0)=\delta^{ij},~~ |a^{ij}(y) -a^{ij}(0)| \leq C(\lambda,\alpha) r_1^{\alpha} |y|^{\alpha}, \quad |b^i| \leq r_1 C(\lambda).
\end{equation}

Assume $D(x, s) \in [\gamma, \gamma+\alpha]$ for any $s\in [r_2, r_1]$. Let $d$ be the largest integer such that $d < \gamma +\alpha$. By Lemma \ref{l:Bound_on_General_DI}, we have $d \leq C(n, \lambda)\Lambda $. We construct the function $\phi(x)$ as follows
    \begin{equation}
        \phi(x) = \int_{B_1} \big( \Gamma(x, y) - \Gamma(0,y)  \big) \Delta \Tilde{u}(y) dy - \sum_{k=1}^d \int_{r_2 /r_1 \leq |y| \leq 1} \Gamma_k(y) P_{y,k}(x) \Delta \Tilde{u}(y) dy 
    \end{equation}

First we prove the growth estimate for the function $\phi$ and that the theorem follows from the claim. 

  \textbf{Claim:} 
  \begin{equation}
    |\phi(x)| \leq 
    \left\{
    \begin{array}{lr}
        C(n, \lambda,\alpha)^{\Lambda} r_1^{\alpha} |x|^{\gamma+\alpha }, &\text{if } |x| \in [r_2/r_1, 1/2].\\
        C(n, \lambda,\alpha)^{\Lambda} r_1^{\alpha} \big(\frac{r_2}{r_1}\big)^{\gamma-1} |x|^{1+\alpha }, &\text{if } |x| \leq r_2/r_1.
    \end{array} 
    \right.
  \end{equation}

The order of growth for $\phi(x)$ when $x$ is small is only $1+\alpha$. However, this will be sufficient for the proof of closeness of doubling index. 

Now we prove the theorem. Suppose the claim be true. Choose $r_0 \leq C(n,\lambda, \alpha)^{-\Lambda} \epsilon^{1/\alpha}$. Then since $\Delta \phi = \Delta \Tilde{u}$ in $B_{1}$, then we can choose $h = \Tilde{u} - \phi$ be harmonic. For any $|y| \leq 1/2$, choose $i$ satisfying $1/2 \leq 2^i|y| < 1$, then by the upper bound of the doubling index, 
\begin{equation}
    \fint_{\partial B_|y|} \Tilde{u}^2 \geq 4^{-(\gamma+\alpha) i} \fint_{\partial B_{2^i |y|}} \Tilde{u}^2 \geq C(n,\lambda,\alpha)^{-\Lambda} |y|^{2(\gamma+\alpha)},
\end{equation}

Therefore,
\begin{equation}\label{e:phi_y_harmonic_appro}
    \begin{split}
        & \quad |h(y) - \Tilde{u}(y)| = |\phi(y)| \\
        & \leq  C(n, \lambda,\alpha)^{\Lambda} r_1^{\alpha} \Big( \fint_{\partial  B_{|y|}} \Tilde{u}^2 \Big)^{1/2} \\
        &\leq \epsilon \Big( \fint_{\partial  B_{|y|}} \Tilde{u}^2 \Big)^{1/2}.
    \end{split}
\end{equation}

Similarly, for any $|y| \leq r_2/r_1$, we have
\begin{equation}\label{e:phi_y_small_harmonic_appro}
    \begin{split}
        & \quad |h(y) - \Tilde{u}(y)| = |\phi(y)| \\
        & \leq C(n, \lambda,\alpha)^{\Lambda} r_1^{\alpha} \big(\frac{r_2}{r_1}\big)^{\gamma-1} |y|^{1+\alpha } \Big( \fint_{\partial  B_{r_2/r_1}} \Tilde{u}^2 \Big)^{1/2} \big(\frac{r_2}{r_1}\big)^{-\gamma - \alpha} \\
        & \leq  C(n, \lambda,\alpha)^{\Lambda} r_1^{\alpha} \Big( \fint_{ \partial B_{r_2/r_1}} \Tilde{u}^2 \Big)^{1/2}\\
       & \leq \epsilon \Big( \fint_{ \partial B_{r_2/r_1}} \Tilde{u}^2 \Big)^{1/2}
    \end{split}
\end{equation}

Hence we prove the first inequality in \ref{e:Harmonic_close}. 

Next we prove the doubling index closeness. Fix $r \geq r_2/r_1$. Note that for any $r_2/r_1 \leq |y| \leq r$, according to \ref{e:phi_y_harmonic_appro}, we have
\begin{equation}\label{e:h_ueps}
    \fint_{\partial B_r} |h - \Tilde{u}|^2 = \fint_{\partial B_r} |\phi(y)|^2 \leq \epsilon^2  \fint_{\partial B_{r}} \Tilde{u}^2.
\end{equation}

Also,
\begin{equation}
    \fint_{\partial B_r} |h + \Tilde{u}|^2 \leq 2 \fint_{\partial B_r} (|h - \Tilde{u}|^2 + 4\Tilde{u}^2 ) \leq 10\fint_{\partial B_r} \Tilde{u}^2.
\end{equation}

Therefore we have
\begin{equation}\label{e:h_u_uniqueness}
    \fint_{\partial B_r} \big| |h^2| - |\Tilde{u}|^2 \big| \leq  \fint_{\partial B_r} |h + \Tilde{u}| |h - \Tilde{u}| \leq \frac{1}{2} (\fint_{\partial B_r} \frac{2|h-\Tilde{u}|^2}{\epsilon} +  \frac{\epsilon}{2} |h+\Tilde{u}|^2) \leq 4\epsilon \fint_{\partial B_r} \Tilde{u}^2.
\end{equation}

This proves that for any $r \in [r_2/r_1, 1/2]$ and $\epsilon < 1/100$,
\begin{equation}
    |D^{h}(0, r) - D^{\Tilde{u}}(0, r) | \leq \log_4 \frac{1+5\epsilon}{1-5\epsilon} \leq 10 \epsilon.
\end{equation}

\textbf{Proof of the Claim: }
    In the following we prove the claim.

    For $x \in B_{1/2}\setminus B_{r_2/r_1}$, we can split the integral into three parts
    \begin{equation}
    \begin{split}
        I_1 &= \int_{B_{2|x|}} \big( \Gamma(x, y) - \Gamma(0, y) \big) \Delta \Tilde{u}(y) dy\\
        I_2 &= -\int_{r_2/r_1 \leq |y| \leq 2|x|} \sum_{k=1}^d \Gamma_k(y) P_{y,k}(x) \Delta \Tilde{u}(y) dy\\
        I_3 &= \int_{2|x| \leq |y| \leq 1} \Big( \Gamma(x,y) - \sum_{k=0}^d \Gamma_k(y) P_{y,k}(x) \Big) \Delta \Tilde{u}(y) dy.
        \end{split}
    \end{equation}

First we deal with $I_1$. Let $\Omega_{\delta} \equiv B_{2|x|} \setminus (B_{\delta} \cup B_{\delta}(x))$ for small $\delta$. Using integration  by parts we have
    \begin{equation}
    \begin{split}
         &\quad I_1 \leq \Big| \int_{\Omega_{\delta}} (\Gamma(x, y) - \Gamma(0, y) ) \Delta \Tilde{u}(y) dy \Big| \\
        &\leq \int_{\Omega_{\delta}} \big|\Gamma(x, y) - \Gamma(0, y) \big| |b^i||\nabla \Tilde{u}(y)| dy \Big| 
       +\int_{\Omega_{\delta}} |\nabla_y (\Gamma(x, y) - \Gamma(0, y) )| \big| a^{ij}(y) - a^{ij}(0) \big| |\nabla \Tilde{u}(y)| dy \\
       & +\int_{\partial B_{2|x|}} \big| |x-y|^{2-n} - |y|^{2-n} \big| \big| a^{ij}(y) - a^{ij}(0) \big|  \big|\nabla \Tilde{u} \big| dS_y + \int_{\partial B_{\delta}} \big| |x-y|^{2-n} - |y|^{2-n} \big| \big| a^{ij}(y) - a^{ij}(0) \big| \big|\nabla \Tilde{u} \big| dS_y  \\
        & + \int_{\partial B_{\delta}(x)} \big| |x-y|^{2-n} - |y|^{2-n} \big| \big| a^{ij}(y) - a^{ij}(0) \big| 
 \big|\nabla \Tilde{u} \big| dS_y   \\
    \end{split}
    \end{equation}

Note that the lower bound on doubling index is bounded below by $\gamma$ and we can apply Lemma \ref{l:Gradient_Growth} to obtain
\begin{equation}
\begin{split}
    &\quad \int_{\Omega_{\delta}} \big| \Gamma(x, y) - \Gamma(0, y) \big| |b^i||\nabla \Tilde{u}(y)| dy \Big| \\
    & \leq C(n, \lambda,\alpha)r_1 \int_{B_{2|x|}} \big| \Gamma(0, y) \big| |\nabla \Tilde{u}(y)|  dy  + C(n, \lambda,\alpha)r_1 \int_{B_{2|x|}} \big| \Gamma(x, y) \big| |\nabla \Tilde{u}(y)| dy \\
    &\leq C(n, \lambda,\alpha)r_1 \int_{0}^{2|x|} r^{2-n} \sup_{B_r}|\nabla \Tilde{u}| r^{n-1} dr  + C(n, \lambda,\alpha)r_1 \int_{0}^{3|x|} r^{2-n} \sup_{B_{|x| +r} \cap B_1}|\nabla \Tilde{u}| r^{n-1} dr \\
    & \leq C(n, \lambda,\alpha)^{\Lambda} r_1 |x|^{\gamma + 1}. 
\end{split}
\end{equation}

Similarly, we have
\begin{equation}
    \int_{\Omega_{\delta}} \big| \nabla_y (\Gamma(x, y) - \Gamma(0, y) ) \big| a^{ij}(y) - a^{ij}(0) \big| |\nabla \Tilde{u}(y)| dy \leq C(n, \lambda,\alpha)^{\Lambda} r_1^{\alpha} |x|^{\gamma + \alpha}.
\end{equation}

Also, 
\begin{equation}
\begin{split}
    & \quad\int_{\partial B_{2|x|}} \big| |x-y|^{2-n} - |y|^{2-n} \big| \big| a^{ij}(y) - a^{ij}(0) \big|  \big|\nabla \Tilde{u} \big| dS_y \\
    & \leq C(n, \lambda,\alpha)^{\Lambda}r_1^{\alpha} \big( |2x|^{2-n} + |x|^{2-n} \big) |2x|^{\alpha} |2x|^{\gamma - 1} |2x|^{n-1}\\
    & \leq C(n, \lambda,\alpha)^{\Lambda} r_1^{\alpha} |x|^{\gamma + \alpha}.
\end{split}
\end{equation}

And
\begin{equation}
    \begin{split}
         & \quad \int_{\partial B_{\delta}} \big| |x-y|^{2-n} - |y|^{2-n} \big| \big| a^{ij}(y) - a^{ij}(0) \big| \big|\nabla \Tilde{u} \big| dS_y \\
         &\leq C(n, \lambda,\alpha)^{\Lambda}r_1^{\alpha} \big( (|x|-\delta)^{2-n} + \delta^{2-n} \big) \delta^{\alpha + (n-1)}\\
         & \leq C(n, \lambda,\alpha)^{\Lambda} r_1^{\alpha} \delta^{1 + \alpha},
    \end{split}
\end{equation}
where we have used the gradient estimate $\sup_{B_{1/2}(0)}|\nabla \tilde{u}|\le C(n,\lambda,\alpha)$.
Similarly,
\begin{equation}
    \begin{split}
         & \quad \int_{\partial B_{\delta}(x)} \big| |x-y|^{2-n} - |y|^{2-n} \big| \big| a^{ij}(y) - a^{ij}(0) \big| 
 \big|\nabla \Tilde{u} \big| dS_y  \\
         &\leq C(n, \lambda,\alpha)^{\Lambda} r_1^{\alpha} \big( \delta^{2-n} + (|x|-\delta)^{2-n}  \big) \delta^{\alpha + (n-1)}\\
         & \leq C(n, \lambda,\alpha)^{\Lambda} r_1^{\alpha} \delta^{1 + \alpha}.
    \end{split}
\end{equation}

Combining all and letting $\delta \to 0$ we have
\begin{equation}
    I_1 \leq C(n, \lambda,\alpha)^{\Lambda} r_1^{\alpha} |x|^{\gamma + \alpha}.
\end{equation}

Next we estimate $I_2$. Using integration by parts we have
\begin{equation}
\begin{split}
    |I_2| &\leq  \sum_{k=1}^d \Big( \int_{r_2/r_1 \leq |y| \leq 2|x|} | \Gamma_k(y)| |P_{y,k}(x)| |\nabla \Tilde{u}(y)||b^{i}(y)| dy  + \int_{|y|=2|x|} |\Gamma_k(y) P_{y,k}(x)| |\nabla \Tilde{u}(y)||a^{ij}(y) - a^{ij}(0)| \\ 
    & + \int_{|y|=r_2/r_1} |\Gamma_k(y) P_{y,k}(x)| |\nabla \Tilde{u}(y)||a^{ij}(y) - a^{ij}(0)| + \int_{r_2/r_1 \leq |y| \leq 2|x|} |\nabla_y \big( \Gamma_k(y) P_{y,k}(x) \big)| |\nabla \Tilde{u}(y)||a^{ij}(y) - a^{ij}(0)| dy \Big).
\end{split}
\end{equation}

By Lemma \ref{l:Gradient_Growth}, Lemma \ref{l:Green_Kernel_growth} and Lemma \ref{l:hhp_C0_bound}, since $\gamma - d > -\alpha > -1$, we have
\begin{equation}
\begin{split}
    & \quad \sum_{k=1}^d \int_{r_2/r_1 \leq |y| \leq 2|x|} | \Gamma_k(y)| |P_{y,k}(x)| |\nabla \Tilde{u}(y)||b^{i}(y)| dy  \\
    &\leq  C(n, \lambda,\alpha)^{\Lambda} r_1 \sum_{k=1}^d \big(\frac{4}{3}\big)^k k^{\frac{n}{2}-1} |x|^k \int_0^{2|x|} r^{2-n-k} r^{\gamma - 1} r^{n-1} dr \\
    & \leq  C(n, \lambda,\alpha)^{\Lambda} r_1 |x|^{\gamma + 1}.
\end{split}
\end{equation}

Similarly,
\begin{equation}
    \quad \sum_{k=1}^d \int_{|y|=2|x|} |\Gamma_k(y) P_{y,k}(x)| |\nabla \Tilde{u}(y)||a^{ij}(y) - a^{ij}(0)| \leq C(n, \lambda,\alpha)^{\Lambda} r^{\alpha}_1 |x|^{\gamma+\alpha }.
\end{equation}
\begin{equation}
    \quad \sum_{k=1}^d \int_{r_2/r_1 \leq |y| \leq 2|x|} |\nabla_y \big( \Gamma_k(y) P_{y,k}(x) \big)| |\nabla \Tilde{u}(y)||a^{ij}(y) - a^{ij}(0)| dy
 \leq C(n, \lambda,\alpha)^{\Lambda} r^{\alpha}_1 |x|^{\gamma+\alpha }.
\end{equation}

Also, since $\gamma+\alpha-d > 0$ and $ r_2/r_1 \leq |x|$, we have
\begin{equation}
\begin{split}
    &\quad \sum_{k=1}^d \int_{|y|=r_2/r_1} |\Gamma_k(y) P_{y,k}(x)| |\nabla \Tilde{u}(y)||a^{ij}(y) - a^{ij}(0)| \\
    & \leq C(n, \lambda,\alpha)^{\Lambda} r^{\alpha}_1  \sum_{k=1}^d \big(\frac{4}{3}\big)^k k^{\frac{n}{2} - 1} |x|^k \big(\frac{r_2}{r_1}\big)^{(2-n-k) + (\gamma -1) + \alpha + n-1} \\
    &\leq C(n, \lambda,\alpha)^{\Lambda} r^{\alpha}_1 |x|^{\gamma+\alpha }.
\end{split}
\end{equation}

This finishes the estimate of $I_2$. For $I_3$, using integration by parts we have
\begin{equation}
\begin{split}
    |I_3| &\leq  \int_{2|x| \leq |y| \leq 1} | R_{y,d}|  |\nabla \Tilde{u}(y)||b^{i}(y)| dy + \int_{|y|=1} |R_{y,d}| |\nabla \Tilde{u}(y)||a^{ij}(y) - a^{ij}(0)| \\ 
    & + \int_{|y|=2|x|} |R_{y,d}| |\nabla \Tilde{u}(y)||a^{ij}(y) - a^{ij}(0)|  + \int_{2|x| \leq |y| \leq 1} |\nabla_y R_{y,d}| |\nabla \Tilde{u}(y)||a^{ij}(y) - a^{ij}(0)| dy .
\end{split}
\end{equation}

By similar calculations and Lemma \ref{l:Gradient_Growth} and noting that $d+1\ge \gamma+\alpha$, we have
 \begin{equation}
     |I_3| \leq C(n, \lambda,\alpha)^{\Lambda} r^{\alpha}_1 |x|^{\gamma+\alpha }. 
 \end{equation}

Suppose $|x| \leq r_2/r_1$. Then we just need to change the gradient estimates to $|\nabla \Tilde{u}|(x) \leq C(n, \lambda,\alpha)^{\Lambda} \big( \frac{r_2}{r_1}\big)^{\gamma -1} $. Following the same steps, we can prove that for any $|x| \leq r_2/r_1$
\begin{equation}
    |\phi(x)| \leq C(n, \lambda,\alpha)^{\Lambda}  r_1^{\alpha} \big( \frac{r_2}{r_1}\big)^{\gamma -1} |x|^{1 + \alpha}.
\end{equation}

The proof of the claim is now finished.
\end{proof}

The first corollary will be the following quantitative uniqueness theorem. Basically it says  the uniform symmetry, see definition \ref{d:uniform_symm}, is true in the interval where the doubling index is pinched. 

\begin{theorem}\label{t:Quan_Unique}
Let $u: B_1 \to \RR$ be a solution to (\ref{e:elliptic_equ}) (\ref{e:assumption}) with doubling assumption \ref{e:DI_bound_assumption}. Let $\epsilon \leq \min\{10^{-3}, \alpha/2\}$. Then there exists some $r_0 = C(n, \lambda, \alpha)^{-\Lambda} \epsilon^{1/\alpha}$ such that if $100 r_2 \leq r_1 \leq r_0$ and $|D(x,s) - D(x,r_1)| \leq \epsilon$ for any $s \in [r_2, r_1]$, then $u$ is \textbf{uniformly $(0, 10\epsilon, x)$-symmetric} in $[4r_2, r_1/4]$.

\end{theorem}

\begin{proof}

According to (\ref{e:h_u_uniqueness}) in the proof of Proposition \ref{p:harmonic_appro}, for any $t \in [r_2/r_1, 1/2]$,
\begin{equation}
    1- 5\epsilon \leq \big( \frac{\fint_{\partial B_t} h^2}{\fint_{\partial B_t} |u_{x,r_1}|^2} \big)^{1/2} \leq 1+5\epsilon
\end{equation}

Using this inequality, we have by scaling that
\begin{equation}
\begin{split}
    & \quad \fint_{\partial B_1} |u_{x,tr_1} - h_{0,t}|^2 \\
    & = \fint_{\partial B_t} \big| \frac{u_{x,r_1}}{(\fint_{\partial B_t} |u_{x,r_1}|^2)^{1/2}} -  \frac{h}{(\fint_{\partial B_t} |u_{x,r_1}|^2)^{1/2}}   +  \frac{h}{(\fint_{\partial B_t} |u_{x,r_1}|^2)^{1/2}}  -\frac{h}{(\fint_{\partial B_t} h^2 )^{1/2}} \big|^2 \\
    &\leq 2 \frac{\fint_{\partial B_t} |u_{x,r_1} - h|^2 }{\fint_{\partial B_t} |u_{x,r_1}|^2} +2 \fint_{\partial B_t} | \frac{5\epsilon h}{(\fint_{\partial B_t} h^2 )^{1/2}}|^2 \\
    & \leq 60 \epsilon^2.
\end{split}
\end{equation}

Moreover, by Proposition \ref{p:harmonic_appro}, the doubling index of $h$ is pinched and it enables us to apply Proposition \ref{p:harm_unique} to $h$. Hence we can find a homogeneous harmonic polynomial $P_d$ such that for $t\in (4r_2/r_1, 1/4)$ we have
\begin{equation}
    \fint_{\partial B_{1}} \big| h_{0, t} - P_d \big|^2 \leq 7\epsilon. 
\end{equation}
Here $P_d$ is independent of $t$.

By H\"older inequality we have
\begin{equation}
    \fint_{\partial B_1} |u_{x,tr_1} - P_d|^2 \leq 21 \fint_{\partial B_1} |u_{x,tr_1} - h_{0,t}|^2 + \frac{11}{10} \fint_{\partial B_1} |h_{0,t} - P_d|^2 \leq 1260 \epsilon^2 + 8\epsilon < 10\epsilon.
\end{equation}

This finishes the proof.
\end{proof}

As a second corollary, we prove more properties for the doubling index which are of fundamental importance in our inductive decomposition arguments. These can be viewed as generalizations to Lemma \ref{l:Harm_pinch_near_Z} and Lemma \ref{l:DI_Drop_Harmonic}. 

\begin{lemma}\label{l:DI_Drop_Elliptic}
    Let $u: B_1 \to \RR$ be a solution to (\ref{e:elliptic_equ}) (\ref{e:assumption}) with doubling assumption \ref{e:DI_bound_assumption}. Let $\epsilon \leq \epsilon_0=\min\{10^{-3}, \alpha/2\}$. There exists some $r_0 =C(n, \lambda, \alpha)^{-\Lambda}\epsilon^{1/\alpha}$ such that
    \begin{enumerate}
        \item If $20r_2 \leq r_1 \leq r_0$ and $|D(x,r_2) - D(x,r_1) | \leq \epsilon$, then there exists some positive integer $d$ such that for any $s \in [4r_2, r_1/4]$
        \begin{equation}
            |D(x,s) - d| \leq 8\epsilon.
        \end{equation}
        \item If $D(x,r_1) \leq d - \epsilon$ with $r_1 \leq r_0$, then for any $s \leq \epsilon^{C(\alpha)} r_1$,
        \begin{equation}
            D(x,s) \leq d -1 + \epsilon.
        \end{equation}
    \end{enumerate}
\end{lemma}

\begin{proof}
    We can apply the Almost Monotonicity Theorem \ref{t:Almost_monotone} with $\epsilon_0=\min\{10^{-3}, \alpha/2\}$ to obtain that
    \begin{equation}
        D(x,s) \leq D(x,r) + \epsilon_0/10
    \end{equation}
    for any $2s \leq r \leq r_0'$. According to the proof of Lemma \ref{l:DI_close_away_from_0} and the claim of Proposition \ref{p:harmonic_appro}, the radius $r_0'$ here can be chosen as $C(n,\lambda,\alpha)^{-\Lambda}$. For any $\epsilon \leq \epsilon_0$, we now set $r_0 =(C(n, \lambda, \alpha)\Lambda)^{-\Lambda}\epsilon^{1/\alpha} < r_0'$ and assume $20 r_2 \leq r_1 \leq r_0$ in the following.

    $(1): $ By Almost Monotonicity Theorem \ref{t:Almost_monotone}, we have 
    \begin{equation}
        D(x,r_2) - \epsilon_0/10 \leq D(x,s) \leq D(x,r_1) + \epsilon_0/10 \text{ for any } s \in [2r_2, r_1/2].
    \end{equation}
    Since $D(x,r_2) \geq D(x,r_1) - \epsilon \geq D(x,r_1) - \epsilon_0$ and $2\epsilon_0 \leq \alpha$, we can apply previous harmonic approximation proposition \ref{p:harmonic_appro}. In particular, there exists some harmonic function $h$ such that if $20r_2 \leq r_1 \leq r_0 = C(n,\lambda,\alpha)^{-\Lambda} \epsilon^{1/\alpha}$,
    \begin{equation}
        |D^{h}(0, s) - D(x,sr_1) | \leq \epsilon \text{ for any } s \in [r_2/r_1, 1/2]. 
    \end{equation}
    By Lemma \ref{l:Harm_pinch_near_Z}, there exists some positive integer $d$ such that $|D^h(0, s) - d| \leq 6\epsilon$ for any $s \in [4r_2/r_1, 1/4]$. The proof of (1) is now finished by triangle inequality again.\\

    $(2): $ Consider $\gamma_1 \equiv \sup\{ r \leq r_1 : D(x,s) \leq d-\epsilon_0 \text{ for any } s \leq r\}. $ If $\gamma_1 \in (0,  \epsilon r_1/10]$, then by Proposition \ref{p:harmonic_appro}, there exists an approximated harmonic function $h_1$ such that $|D^{h_1}(0, s) - D(x,s r_1)| \leq \epsilon/10$ for any $s \in [\gamma_1/r_1, 1/2]$. Then by Lemma \ref{l:DI_Drop_Harmonic} and monotonicity of doubling index for harmonic functions, we have 
    \begin{equation}
         D(x,\frac{\epsilon r_1}{4}) \leq D^{h_1}(0, \frac{\epsilon }{4}) + \frac{\epsilon}{5} \leq d - 1 +\frac{\epsilon}{2} + \frac{\epsilon}{5} = d - 1 + \frac{7\epsilon}{10}.
    \end{equation}

    However, $D(x,\frac{\epsilon r_1}{4}) \geq D(x,\gamma_1) - \epsilon_0 = d - 2\epsilon_0$ by almost monotonicity theorem \ref{t:Almost_monotone}. This is a contradiction. Hence we must have $\gamma_1 \geq \epsilon r_1/10$. 

    Now we have $D(x,s) \leq d- \epsilon_0$ for any $s \leq \gamma_1$. Consider $\gamma_2 \equiv \sup\{ r \leq \gamma_1 : D(x,s) \leq d- 2\epsilon_0 \text{ for any } s \leq r\}.$  If $\gamma_2 \in (0, \epsilon \gamma_1 /10]$, then we can find an approximated harmonic function $h_2$ by Proposition \ref{p:harmonic_appro}. By the same arguments, we can obtain a contradiction and thus $\gamma_2 \geq \frac{\epsilon \gamma_1}{10} \geq (\frac{\epsilon}{10})^2 r_1$.

    We can iterate this arguments for at most $n$ times to obtain $\gamma_n$ with $n \leq \frac{10}{\epsilon_0} = C(\alpha)$ and $\gamma_n \geq (\frac{\epsilon}{10})^n r_1$ as the contradiction argument is valid provided $D(x,\gamma_N) \geq d-1 +\epsilon$. Hence it implies that $D(x,\epsilon^{C(\alpha)}r_1) \leq d -1 + \epsilon$. This finishes the proof.
\end{proof}

Recall that the harmonic function is close to a linear function when its doubling index is close to 1. We generalize this fact to general elliptic solutions, which will be crucial for the iteration arguments in the proof of volume estimates on critical sets. 

\begin{proposition}\label{p:Elliptic_gradient_lower_bound}
    Let $u: B_2 \to \RR$ be a solution to (\ref{e:elliptic_equ}) (\ref{e:assumption}) with doubling assumption \ref{e:DI_bound_assumption}. Let $x \in B_1$. If $\epsilon \leq \epsilon_0= c(n, \alpha,\tau)^{\Lambda}$ and $r \leq r_0 =   C(n,\lambda, \alpha)^{-\Lambda}\epsilon^{1/\alpha}$, then if $D(x,r) \leq 1 + \epsilon$, then
    \begin{equation}
        x \notin C(u).
    \end{equation}
\end{proposition}

\begin{proof}
    According to Proposition \ref{p:harmonic_appro}, there exists a linear function $L$ such that $|L - u_{x,r}|_{C^1(B_1)} \leq \epsilon$. Hence $|\nabla u_{x,r}(0)| \geq 1 - \epsilon$. This proves that $x \notin C(u)$.
\end{proof}

\section{Cone Splitting}

In this section, we will investigate the cone splitting principle for elliptic solutions. We start from the splitting behavior of homogeneous polynomials.

\begin{lemma}\label{l:Polynomial_Split}
    Let $P$ be a polynomial of order $d$. If $P$ is $k$-symmetric with respect to a $k$-dimensional subspace $V$ and $P - P(x)$ is $0$-symmetric with respect to $x \notin V$. Then $P$ is $(k+1)$-symmetric with respect to span$(x, V)$.    
\end{lemma}

\begin{proof}
    We assume $V = \{0\}$ and similar arguments hold for general case. Without loss of generality, we assume $P(0)=0$ and $x=(t, 0, ..., 0)$ with $t\neq 0$. It suffices to prove that $\partial_1 P \equiv 0$.  
    
    Since $P$ is homogeneous with respect to $0$, hence 
    \begin{equation}
        d \cdot P(y) = \nabla P|_y \cdot y,
    \end{equation}
     where $d$  is the degree of $P$.
    Since $P$ is $0$-symmetric with respect to $x$, we have
    \begin{equation}
        d \cdot (P(y) - P(x)) = \nabla P|_y \cdot (y - x). 
    \end{equation}
    Combining these two equations, we have for any $y$
    \begin{equation}
        t \partial_1 P(y) = \nabla P|_y \cdot x = d \cdot P(x).
    \end{equation}

    Hence $\partial_1 P(y)$ is constant. We claim that $P(x)=P(0)=0$. To see this, noting that $\partial_1 P(y)=d\cdot P(x)/t$, we get $P(x)-P(0)=\int_0^{t}\partial_1P(s,0,...,0)ds=d\cdot P(x)$. Thus $P(x)=0$ or $d=1$. If $P(x)=0$, then $\partial_1 P(y)=0$ and we are done. If $d=1$, then we can write $P(y) = \sum_i a_i y_i$. Since $P$ is $0$-symmetric at $x$, we can write $P(y)$ as
    \begin{equation*}
        \sum_{i=1}^n a_i y_i = b_1(y_1 - t) + \sum_{i=2}^n b_i y_i. 
    \end{equation*}
    By comparing the coefficients, we have $a_1 = b_1 = 0$. Hence $\partial_1 P = 0$ and $P(x)=0$. This finishes the proof. 
\end{proof}

\subsection{Cone-Splitting for Harmonic Functions}

Recall that the constancy of doubling index implies the homogeneity of harmonic functions. As in Lemma \ref{l:Polynomial_Split}, if the doubling index of the polynomial $P$ is constant at $x$ and $y$, then it is $1$-symmetric with respect to the direction $y-x$. One can easily obtain the almost cone splitting lemma by assuming that the doubling indices are pinched. 

Our goal is to generalize the almost cone splitting lemma to the general elliptic solutions. As Theorem \ref{t:Quan_Unique} shows, the solution $u$ is almost uniformly $0$-symmetric in the interval where the doubling index is pinched. In the following theorem, we prove that the pinched doubling indices at two points implies the splitting. First we prove the almost cone-splitting lemma for harmonic functions. The results essentially follow \cite{NV}.

\begin{lemma}\label{l:ConeSplitting_Harmonic}
    Let $u: B_{10d} \to \RR$ be a harmonic function and $d \geq 1$ be an integer. Suppose we have
    \begin{enumerate}
        \item $|D(0, s) - d| \leq \epsilon $ for any $s \in [1/10, 9d]$.
        \item $|D(x, s) - d| \leq \epsilon $ for some $x \in B_1$ and for any $s \in [1/10, 9d]$.
    \end{enumerate}
   If $\epsilon \leq \epsilon_0(n)$ for some $\epsilon_0(n) > 0$, then we have
    \begin{enumerate}
        \item Let $P_{0,d}$ and $P_{x,d}$ be the $d$-th degree part of the expansion of $u$ at $0$ and $x$ respectively. Then
        \begin{equation}
            ||P_{0,d} - P_{x,d}|| \leq C(n) \sqrt{\epsilon} |x| ||P_{0,d}||.
        \end{equation}
        \item $x$ is an almost invariant direction of $P_{0,d}$, i.e.
        \begin{equation}
            || x \cdot \nabla P_{0,d}|| \leq C(n) \sqrt{\epsilon} ||\nabla P_{0,d}||.
        \end{equation}
        \item Suppose $|x| \geq \tau$ for some constant $\tau > 0$. If $\epsilon \leq C(n)d^{-4}\tau^4$, then $u$ is uniformly $(1, \sqrt{\epsilon}, 0)$-symmetric in $[1/3, 3d]$ with respect to span$<x>$. 
    \end{enumerate}
\end{lemma}

\begin{proof}
    For the proof of (1) and (2), see Lemma 3.22 in \cite{NV}. We will prove (3) in the following. Without loss of generality, we assume $x = (t, 0, ..., 0)$ with $t \geq \tau$. By Proposition \ref{p:harm_unique}, for any $s \in [1/3, 3d]$, we have
    \begin{equation}
        \fint_{\partial B_1} |u_{0,s} - \frac{P_{0,d}}{||P_{0,d}||}|^2 \leq 7 \epsilon.
    \end{equation}
    
    Now we decompose 
    \begin{equation}
        P_{0,d} = P_1 + P_2,
    \end{equation}
    where $P_1 \in \cP_d(\cancel{x_1})$ and $P_2 \in \cP_d(\cancel{x_1})^{\perp}$. Hence we have $ \partial_1 P_{0,d} = \partial_1 P_2$. By (2) and Lemma \ref{l:Invariant_x_1}, we have
    \begin{equation}
        ||P_2|| \leq ||\partial_1 P_2|| = ||\partial_1 P_{0,d} || \leq C(n) t^{-1} \sqrt{\epsilon d(2d+n-2)} ||P_{0,d}||.
    \end{equation}

Then by triangle inequality, if $\epsilon \leq C(n)d^{-4}\tau^4$, we have
\begin{equation}
    ||\frac{P_{1}}{||P_{1}||}  - \frac{P_{0,d}}{||P_{0,d}||}  ||^2 = 2  - \frac{2 ||P_1||}{||P_{0,d}||}  \leq 2 (1 - \sqrt{1 -C(n) t^{-2} \epsilon d(2d+n-2) } ) \leq \sqrt{\epsilon}/4. 
\end{equation}

Therefore,
\begin{equation}
\begin{split}
    & \quad \big( \fint_{\partial B_1} |u_{0,s} -  \frac{P_{1}}{||P_{1}||}|^2   \big)^{1/2} \\
    & \leq \big( \fint_{\partial B_1} |u_{0,s} -  \frac{P_{0,d}}{||P_{0,d}||}|^2   \big)^{1/2} + \big( \fint_{\partial  B_1} |\frac{P_{1}}{||P_{1}||}  - \frac{P_{0,d}}{||P_{0,d}||}  |^2   \big)^{1/2} \\
    & \leq \sqrt{7 \epsilon} + \epsilon^{1/4}/2 \\
    & \leq \epsilon^{1/4}. 
\end{split}
\end{equation}

Since $P_1$ is $1$-symmetric, the proof is now finished.     
\end{proof}

Next we generalize the previous cone-splitting lemma to higher dimension. First we define the set of points where the doubling index of the harmonic function $h$ is pinched to be
\begin{equation}
    \cV^h_{\epsilon, d, r} (x) \equiv \{ y \in B_{r}(x) : |D(y, s) - d| \leq \epsilon \text { for any } s \in [ 1/10 r, 9d r] \}
\end{equation}

\begin{definition}[$(k, \tau)$-independent]
We say a subset $S \subset B_r$ is $(k, \tau)$-independent in $B_r$ if for any affine $(k-1)$-plane $L$ there exists some point $x \in S$ such that $d(x, L) \geq \tau r$.
\end{definition}

First we recall a lemma about almost $(n-2)$ invariant polynomials. 

\begin{lemma}\label{l:n-2_Invariant_hhp_decomposition}
    Let $P: \RR^n \to \RR$ be a homogeneous harmonic polynomial of degree $d$ such that 
    \begin{equation}
    ||\partial_i P||^2 \leq \epsilon ||\nabla P||^2 \text{ for any } i=1, ..., k.
    \end{equation}
    There exists $C(n)$ such that if $\epsilon \leq C(n)d^{-4k}$, then we can write
    \begin{equation}
        P = P_1 + P_2
    \end{equation}
    where $P_1, P_2 \in \cP_d$ and $P_1$ is $x_1, ..., x_k$-invariant with $||P_1|| \geq (1 - \frac{\sqrt{\epsilon}}{8})||P||$.
\end{lemma}

\begin{proof}
See the proof of Lemma 3.28 \cite{NV}.
\end{proof}

\begin{proposition}\label{p:Harmonic_k_splitting}
    Let $h: B_{10d} \to \RR$ be a harmonic function and $d \geq 1$ be an integer. Fix $\tau \in (0,1)$. Then there exists some $\epsilon_0 = C(n, \tau)d^{-4k}$ such that the following holds. If $\cV^h_{\epsilon,d,1}(0)$ is $(k, \tau)$-independent in $B_{1}$ with $\epsilon \leq \epsilon_0$, then $u$ is uniformly $(k, \sqrt{\epsilon}, y)$-symmetric in $[1/3, 3d ]$ for each point $y \in \cV^h_{\epsilon,d,1}(0)$.
\end{proposition}

\begin{proof}
    Pick any $y \in \cV^h_{\epsilon,d,1}(0)$. Without loss of generality we may assume $y=0$. Since  $\cV^h_{\epsilon,d,1}(0)$ is $(k, \tau)$-independent, there exists $k$ vectors $y_1, ..., y_k$ such that $||y_i|| \geq \tau$ and 
    \begin{equation}
        ||y_i \cdot \nabla P_{0,d}|| \leq C(n) \sqrt{\epsilon}||\nabla P_{0,d}|| \quad \text{ and } \quad y_i \notin B_{\tau}(\text{span}(y_1, ..., y_{i-1})).
    \end{equation}
    
    By Gram–Schmidt process, we can find $k$ orthonormal unit vectors $z_i$ such that 
    \begin{equation}
        ||z_i \cdot \nabla P_{0,d}|| \leq C(n, \tau) \sqrt{\epsilon}||\nabla P_{0,d}||.
    \end{equation}

 By Lemma \ref{l:n-2_Invariant_hhp_decomposition}, if $\epsilon \leq c(n, \lambda)d^{-4k}$, we can find some $P_1$ which is $z_1, ..., z_k$-invariant with $||P_1|| \geq  (1 - \frac{\sqrt{\epsilon}}{8})||P_{0,d}||$. Hence
 \begin{equation}
     ||\frac{P_{1}}{||P_{1}||}  - \frac{P_{0,d}}{||P_{0,d}||} ||^2  = 2 - \frac{2||P_1||}{||P_{0,d}||} \leq \frac{\sqrt{\epsilon}}{4}.
 \end{equation}

Then by Proposition \ref{p:harm_unique}, for any $s \in [1/3, 3d]$ we have
\begin{equation}
\begin{split}
    & \quad \big( \fint_{\partial B_1} |u_{0,s} -  \frac{P_{1}}{||P_{1}||}|^2   \big)^{1/2} \\
    & \leq \big( \fint_{\partial  B_1} |u_{0,s} -  \frac{P_{0,d}}{||P_{0,d}||}|^2   \big)^{1/2} + \big( \fint_{\partial  B_1} |\frac{P_{1}}{||P_{1}||}  - \frac{P_{0,d}}{||P_{0,d}||}  |^2   \big)^{1/2} \\
    & \leq \sqrt{7 \epsilon} + \frac{\epsilon^{1/4}}{2} \\
    & \leq \epsilon^{1/4}. 
\end{split}
\end{equation}
 This finishes the proof.    
\end{proof}

An immediate consequence of the proposition is that the pinched point can only be lying around an $(n-2)$-dimensional plane if the doubling index $d \geq 2$.

\begin{corollary}\label{c:Harmonic_Pinch_pt_near_n-2}
Let $h: B_{10d} \to \RR$ be a harmonic function and fix  $\tau \in (0,1)$. If $d \geq 2$, then there exists some $\epsilon_0 = C(n, \tau)d^{-4n+4}$ such that there exists some at most $(n-2)$-dimensional subspace $V$ such that $\cV^h_{\epsilon,d,1}(0) \subset B_{\tau}(x + V)$ for any $x \in \cV^h_{\epsilon,d,1}(0)$ with $\epsilon \leq \epsilon_0$. 
\end{corollary}

\begin{proof}
    Choose the largest $k$ such that $\cV^h_{\epsilon,d,1}(0)$ is $(k, \tau)$-independent. Then according to the previous Proposition, $u$ is uniformly $(k, \sqrt{\epsilon}, y)$-symmetric in $[1/3, 3d]$ for any $y \in \cV^h_{\epsilon,d,1}(0)$. It suffices to prove that $k \leq n-2$. Suppose $k = n-1$. Then in particular, there exists a normalized linear function $L$ such that $\fint_{\partial B_1} |L - h_{y, 1}|^2 \leq \sqrt{\epsilon}.$ Therefore we have
    \begin{equation}
    \big( \frac{\fint_{\partial B_1} |h_{y,1}|^2}{\fint_{\partial B_{1/2}} |h_{y,1}|^2} \big)^{1/2} \leq \frac{(\fint_{\partial B_1} |L|^2)^{1/2} + (\fint_{\partial B_1} |h_{y,1} - L|^2)^{1/2} }{(  (\fint_{\partial B_{1/2}} |L|^2)^{1/2} - \fint_{\partial B_{1/2}} |h_{y,1} - L|^2)^{1/2}} \leq \frac{1 + \epsilon^{1/4}}{\frac{1}{2} - 2^{n/2} \epsilon^{1/4}}.
    \end{equation}

If $\epsilon \leq C(n)$, we can conclude that $D^h(y,1) \leq 3/2$. This is a contradiction to $d \geq 2$. When $k = n$, we can obtain $D^h(y,1) \leq 1/2$ by replacing the linear function $L$ by constant function $C \equiv 1$ in the arguments as above. This is a contradiction and finishes the proof.
\end{proof} 

\begin{remark}\label{r:V_invariant_scale_point}
    According to Proposition \ref{p:Harmonic_k_splitting}, the uniform symmetry implies that the subspace $V$ here does not depend on the scale where the doubling index is still pinched. Indeed, $V$ does not depend on $x \in \cV^h_{\epsilon,d,1}$ either. This could be seen from the fact that $V$ is the set of almost invariant directions of the $d$-th expansion of $h$. See also Proposition 3.24 and Corollary 3.26 in \cite{NV}.
\end{remark}

Next we want to study the critical set for almost $(n-2)$-symmetric harmonic function. As usual, first let us look at the $(n-2)$-symmetric polynomial $P(x_1,x_2)$ in $\RR^{2} \times \RR^{n-2}$. The standard theory for homogeneous harmonic polynomials with two variables tells us that $|\nabla P(x_1, x_2)| = c(n)||P||_{L^2(\partial B_1)}|x|^{d-1}$. Hence the critical point $C(P) \subset \{0\} \times \RR^{n-2}$. In fact the similar result holds for almost $(n-2)$-symmetric harmonic functions.

\begin{lemma}\label{l:Harmonic_no_critical_pt_away_from_n-2}
    Let $h: B_{5} \to \RR$ be a harmonic function and fix  $\tau \in (0,1)$. Assume $h$ is $(n-2, \epsilon, 3, 0)$-symmetric with respect to $V$ and the approximated homogeneous polynomial is of degree no larger than $d$. If $\epsilon \leq (c(n)\tau)^{2d-2}$, then $|\nabla u(x)| \geq \tau^d$ for $x \in B_1 \setminus B_{\tau}(V)$. In particular, $C(u) \cap B_1 \subset B_{\tau}(V)$. 
\end{lemma}

\begin{proof}
    See the proof of Proposition 3.32 in \cite{NV}. 
\end{proof}

\subsection{Cone-Splitting for Elliptic Solutions}

Next we focus on the almost cone-splitting lemma for general elliptic solutions. This is subtler as the linear transformations $a^{ij}(x)$ varies with points. Recall that in the previous harmonic approximation lemmas \ref{l:Harm_Approx} and \ref{p:harmonic_appro}, we only derive the doubling index closeness at the point $0$. In the following Proposition, we will see that the doubling index closeness also holds at nearby points and at large scales. 

\begin{proposition}\label{p:DI_close_at_nearby_pt}
    Let $u: B_2 \to \RR$ be a solution to (\ref{e:elliptic_equ}) (\ref{e:assumption}) with doubling assumption \ref{e:DI_bound_assumption}.  Let $\epsilon \leq \min\{10^{-3}, \alpha/2\}$ and $r_1 \leq r_0 =   C(n,\lambda, \alpha)^{-\Lambda}\epsilon^{1/\alpha}$. Suppose $|D(x, s) - d| \leq \epsilon$ for any $s \in[r_2, r_1]$ with $20r_2 \leq r_1 \leq r_0$. Let $h$ be the harmonic approximation of $\Tilde{u} = u_{x,r_1}$ as in Proposition \ref{p:harmonic_appro}. For any $\Bar{z} \in B_{(1+\lambda)^{-1/2}r_1/10}(x)$, we define $z = A_x^{-1} (\frac{\Bar{z} - x}{r_1}) \in B_{1/10}$. Then we have the following convergence result for doubling index
    \begin{equation}
        |D^{u}({\Bar{z}}, sr_1) - D^h(z, s)| \leq \epsilon \text{ for any } s \in [1/10, 1/4].
    \end{equation}
 \end{proposition}

\begin{proof}
To prove the result, we will prove the inequality holds at one scale $1/4$ for simplicity. The same arguments hold for other scales. 

By triangle inequality we have
\begin{equation}
    |D^{u}({\Bar{z}}, r_1/4) - D^h(z, 1/4)| \leq |D^{u}({\Bar{z}}, r_1/4) - E | + | E - D^h(z, 1/4)|.
\end{equation}
where we define $E$ as
\begin{equation}
    E \equiv \log_4\frac{\fint_{\partial B_{1/2}(z)}|\Tilde{u} - \Tilde{u}(z)|^2  }{\fint_{\partial B_{1/4}(z)}|\Tilde{u} - \Tilde{u}(z)|^2 }.
\end{equation}

We will prove both parts are controlled by small $\epsilon$.

\textbf{Step 1: $| E - D^h(z, 1/4)| \leq \epsilon.$}

It suffices to check that
\begin{equation}
       \Big| \frac{\fint_{\partial  B_{1/2}(z)} |\Tilde{u} - \Tilde{u}(z)|^2 }{ \fint_{\partial B_{1/2}(z)} |h - h(z)|^2} - 1 \Big|+   \Big| \frac{\fint_{\partial  B_{1/4}(z)} |\Tilde{u} - \Tilde{u}(z)|^2 }{ \fint_{\partial B_{1/4}(z)} |h - h(z)|^2} - 1 \Big| \leq \epsilon.   
\end{equation}
We only consider the case 
\begin{align}
    \Big| \frac{\fint_{\partial  B_{1/4}(z)} |\Tilde{u} - \Tilde{u}(z)|^2 }{ \fint_{\partial B_{1/4}(z)} |h - h(z)|^2} - 1 \Big| \leq \epsilon.  
\end{align}

Since $\fint_{B_{1/8}(0)}h=0$, we have for any $z\in B_{1/10}(0)$ that 
\begin{align}
\fint_{B_{1/8}(0)}h^2= \inf_{a\in \mathbb{R}}\fint_{B_{1/8}(0)}|h-a|^2&\le \fint_{B_{1/8}(0)}|h-h(z)|^2\\
 &\le C(n)\fint_{B_{1/4}(z)}|h-h(z)|^2\le C(n)\fint_{\partial B_{1/4}(z)}|h-h(z)|^2.
\end{align}
By mean value inequality and the monotonicity $\fint_{\partial B_r(0)} h^2\le\fint_{\partial B_s(0)} h^2$ for any $r\le s$,  this implies that 
\begin{align}
\fint_{\partial B_{1/10}(0)} h^2 \leq C(n)\fint_{\partial B_{1/4}(z)} |h - h(z)|^2.
\end{align}

     Since $h$ is the harmonic approximation for $\Tilde{u}$, by the claim in the proof of Proposition \ref{p:harmonic_appro} and the elliptic estimates, if $z \in B_{1/10}$,
     \begin{equation}
     \begin{split}         
    & \quad \Big|\fint_{\partial  B_{1/4}(z)} |\Tilde{u} - \Tilde{u}(z)|^2 -\fint_{\partial B_{1/4}(z)} |h - h(z)|^2 \Big| \\
    & \leq C(n,\lambda, \alpha)^\Lambda r_1^{\alpha} \fint_{\partial B_1} h^2 \\
    & \leq C(n,\lambda, \alpha)^\Lambda 4^{4C(n,\lambda)\Lambda} r_1^{\alpha} \fint_{\partial B_{1/10}} h^2 \\
    & \leq C(n,\lambda, \alpha)^\Lambda  r_1^{\alpha}  \fint_{\partial B_{1/4}(z)} |h - h(z)|^2\\
    & \leq \frac{\epsilon}{10} \fint_{\partial B_{1/4}(z)} |h - h(z)|^2. 
     \end{split}
     \end{equation}

This finishes the proof of step 1.

\textbf{Step 2: $| D^{u}({\Bar{z}}, r_1/4) - E | \leq \epsilon.$}

 Recall that
\begin{equation}
    D^{u}({\Bar{z}}, r_1/4) = \text{log}_4 \frac{\fint_{ \partial B_{1/2}} (u(\Bar{z} +r_1 A_{\Bar{z}}(y)) - u(\Bar{z}))^2 dy}{\fint_{\partial  B_{1/4}} (u(\Bar{z} +r_1A_{\Bar{z}}(y)) - u(\Bar{z}))^2 dy}.
\end{equation}
and
\begin{equation}
    E \equiv \log_4\frac{\fint_{\partial B_{1/2}(z)}|\Tilde{u} - \Tilde{u}(z)|^2  }{\fint_{\partial B_{1/4}(z)}|\Tilde{u} - \Tilde{u}(z)|^2 } = \log_4\frac{\fint_{\partial B_{1/2}(z)}|u(x +r_1A_x(y)) - u(x +r_1A_x(z))|^2 dy  }{\fint_{\partial B_{1/4}(z)}|u(x +r_1A_x(y)) - u(x +r_1A_x(z)) |^2 dy }.
\end{equation}

Since $\Bar{z}=x + r_1 A_x(z) $, it suffices to prove that
\begin{equation}
       \Big| \frac{\fint_{ \partial B_{1/2}} (u(\Bar{z} +r_1A_{\Bar{z}}(y)) - u(\Bar{z}))^2 dy}{ \fint_{\partial  B_{1/2}(z)} (u(x +r_1A_x(y)) - u(\bar{z}))^2 dy} - 1 \Big| +\Big| \frac{\fint_{ \partial B_{1/4}} (u(\Bar{z} +r_1A_{\Bar{z}}(y)) - u(\Bar{z}))^2 dy}{ \fint_{\partial  B_{1/4}(z)} (u(x +r_1A_x(y)) - u(\bar{z}))^2 dy} - 1 \Big| \leq \epsilon.
\end{equation}
We will only prove the case 
\begin{equation}
\Big| \frac{\fint_{ \partial B_{1/4}} (u(\Bar{z} +r_1A_{\Bar{z}}(y)) - u(\Bar{z}))^2 dy}{ \fint_{\partial  B_{1/4}(z)} (u(x +r_1A_x(y)) - u(\bar{z}))^2 dy} - 1 \Big| \leq \epsilon.
\end{equation}
First, we have 
\begin{align}
  \fint_{\partial  B_{1/4}(z)} (u(x +r_1A_x(y)) - u(\bar{z}))^2 dy&=   \fint_{\partial  B_{1/4}(z)} (u(\bar{z} +r_1A_x(y-z)) - u(\bar{z}))^2 dy\\
  &=\fint_{\partial  B_{1/4}(0)} (u(\bar{z} +r_1A_x(y)) - u(\bar{z}))^2 dy.
\end{align}
Then by elliptic estimates and bound on doubling index we have
\begin{equation}
\begin{split}
   & \quad \Big| \fint_{\partial  B_{1/4}} (u(\Bar{z} +r_1A_{\Bar{z}}(y)) - u(\Bar{z}))^2 dy - \fint_{ \partial B_{1/4}} (u(\bar{z} +r_1A_x(y)) - u(\Bar{z}))^2 dy\Big| \\
   & = \Big| \fint_{\partial  B_{1/4}} |u(\Bar{z} +r_1A_{\Bar{z}}(y)) - (u(\bar{z}+r_1A_x(y))| | u(\Bar{z} +r_1A_{\Bar{z}}(y)) + u(\bar{z} +r_1A_x(y))  - 2 u(\Bar{z})| dy \Big| \\
   & \leq C(n,\lambda,\alpha) ( r_1 \sup_{B_{1/2}}|\nabla u|) \fint_{\partial B_{1/4}} |u(\Bar{z} +r_1A_{\Bar{z}}(y)) - u(\bar{z} +r_1A_x(y))| dy \\
   & \le C(n,\lambda,\alpha) ( r_1 \sup_{B_{1/2}}|\nabla u|)^2\fint_{\partial B_{1/4}} |A_{\Bar{z}}(y) - A_{x}(y)| dy\\
   & \leq C(n,\lambda, \alpha)^{\Lambda} r_1^{\alpha}  \fint_{ \partial B_{1/4}} (u(\bar {z} +r_1A_x(y)) - u(\Bar{z}))^2 dy.
\end{split} 
\end{equation}

Hence by choosing $C(n,\lambda, \alpha)^{\Lambda} r_1^\alpha \leq \epsilon$ we have
\begin{equation}
    \big| \fint_{ B_{1/4}} (u(\Bar{z} +r_1A_{\Bar{z}}(y)) - u(\Bar{z}))^2 dy - \fint_{ B_{1/4}} (u(\bar {z} +r_1A_x(y)) - u(\Bar{z}))^2 dy \big| \leq \epsilon \fint_{ B_{1/4}} (u(\bar {z} +r_1A_x(y)) - u(\Bar{z}))^2 dy.
\end{equation}

This proves that step 2 and the theorem follows.
\end{proof}

Now we can prove the cone-splitting lemma for general elliptic solutions, which can be viewed as a generalization of Proposition \ref{p:Harmonic_k_splitting}. First we define the set of points where the doubling index of the general elliptic solution $u$ is pinched to be
\begin{equation}
    \cV^u_{\epsilon, d, r} (x) \equiv \{ y \in B_{r}(x) : |D(y,s) - d| \leq \epsilon \text { for any } s \in [ r/100 , C_2(\lambda)\Lambda r] \},
\end{equation}
where $C_2(\lambda) = 10^{3}(1+\lambda)^{1/2}$. The constants here are a bit different from that in harmonic case. As we mainly deal with general elliptic solutions, this would not cause any confusion.  

Finally we can prove the cone-splitting lemma for general elliptic solutions.

\begin{proposition}
\label{p:Cone_Split_Elliptic}
    Let $u: B_2 \to \RR$ be a solution to (\ref{e:elliptic_equ}) (\ref{e:assumption}) with doubling assumption \ref{e:DI_bound_assumption}. Fix $\tau < c_3(n,\lambda)$ for some small constant $c_3(n,\lambda)$. If $\epsilon \leq \epsilon_0= C(n, \alpha,\tau)\Lambda^{-4n}$ and $r \leq r_0 =   C(n,\lambda, \alpha)^{-\Lambda}\epsilon^{1/\alpha}$, then the following is true. If $\cV^u_{\epsilon,d,r}(x)$ is $(k, \tau)$-independent in $B_{r}(x)$ for $x \in B_1$, then $u$ is uniformly $(k, \sqrt{\epsilon}, z)$-symmetric in $[r/3, 3 \Lambda r]$ for each point $z \in \cV^u_{\epsilon,d,r}(x)$.
    
\end{proposition}

\begin{proof}
Without loss of generality we may assume $x \in \cV^u_{\epsilon,d,r}(x)$. We will prove that $u$ is uniformly $(k, \sqrt{\epsilon}, x)$-symmetric in $[r/3, 3 \Lambda r]$. 

Now fix $\Bar{y} \in \cV^u_{\epsilon,d,r}(x)$. Choose $r_1 = C_2(\lambda)\Lambda r$ and define $y = A_x^{-1}(\frac{\Bar{y} - x}{r_1}) \in B_{1/10}$.   Hence by Proposition \ref{p:DI_close_at_nearby_pt}, there exists an approximated harmonic function $h$ such that for any $s \in [1/10, 1/4]$
\begin{equation}
    |D^u({\Bar{y}},s r_1) - D^h(y, s)| \leq \epsilon.
\end{equation}

Hence we can conclude that $|D^h(y, s) - d| \leq 2\epsilon$ for any $s \in [1/10, 1/4]$. 

We can also choose other $r_1'$, say $r_1' = r_1/2 = C_2(\lambda)\Lambda r/2$. Then $y' = A_x^{-1}(\frac{\Bar{y} - x}{r_1'}) \in B_{1/10}$. Then there exists some approximated harmonic function $h'$ such that the Proposition \ref{p:DI_close_at_nearby_pt} also holds for $h'$ at $r_1'$. Indeed we have $|D^{h'}(y', s) - d| \leq 2\epsilon$ for any $s\in [1/10, 1/4]$. However, according to Proposition \ref{p:harmonic_appro} and Theorem \ref{t:Quan_Unique}, we may choose $h'(x) = c h(x/2)$ for some scaling constant $c$ by the uniqueness result. By the scaling invariance property of doubling index for $h'$, we have $|D^{h'}(y', s) - d| \leq 2\epsilon$ for any $s\in [1/10, 1/2]$.

We can continue this process and eventually we choose $\Tilde{r}_1 = 10(1+\lambda)^{1/2} r$. Then by the same arguments we can conclude that there exists some approximated harmonic function $\Tilde{h}$ such that $|D^{\Tilde{h}}(\Tilde{y}, s) - d| \leq 2\epsilon$ for any $s \in [1/10, 10\Lambda]$ with $\Tilde{y} = A_x^{-1}(\frac{\Bar{y} - x}{\Tilde{r}_1}) \in B_{1/10}$. This implies that $\Tilde{y} \in \cV^{\Tilde{h}}_{2\epsilon, d, 1}$. Since $\cV^u_{\epsilon,d,r}(x)$ is $(k, \tau)$-independent in $B_r(x)$, we have $\cV^{\Tilde{h}}_{2\epsilon, d, 1}$ is $(k, c_3^{-1}(n,\lambda)\tau)$-independent in $B_1$. Applying the cone-splitting lemma \ref{p:Harmonic_k_splitting} to the harmonic function $\Tilde{h}$, we have $\Tilde{h}$ is uniformly $(k, \sqrt{\epsilon}/2, 0)$-symmetric in $[1/3, 3\Lambda]$ provided $\epsilon \leq C(n, \lambda)\Lambda^{-4n}$. By harmonic approximation proposition \ref{p:harmonic_appro}, we can conclude that $u$ is uniformly $(k, \sqrt{\epsilon}, x)$-symmetric in $[r/3, 3\Lambda r]$.
\end{proof}

Similar to Corollary \ref{c:Harmonic_Pinch_pt_near_n-2}, we have 

\begin{corollary}\label{c:Elliptic_Pinch_pt_near_n-2}
Let $u: B_2 \to \RR$ be a solution to (\ref{e:elliptic_equ}) (\ref{e:assumption}) with doubling assumption \ref{e:DI_bound_assumption}. Fix $\tau < c_3(n,\lambda)$ for some small constant $c_3(n,\lambda)$. If $\epsilon \leq \epsilon_0= C(n, \alpha,\tau)\Lambda^{-4n}$ and $r \leq r_0 =   C(n,\lambda, \alpha)^{-\Lambda}\epsilon^{1/\alpha}$, then the following is true. 
If $d \geq 2$, then there exists some at most $(n-2)$-dimensional subspace $V$ such that $\cV^u_{\epsilon,d,r}(x) \subset B_{\tau r}(y + V)$ for any $y \in \cV^u_{\epsilon,d,r}(x)$. 
\end{corollary}

\begin{proof}
    The proof follows from Proposition \ref{p:Cone_Split_Elliptic} and same arguments as in \ref{c:Harmonic_Pinch_pt_near_n-2}.
\end{proof}
Again, similar to Remark \ref{r:V_invariant_scale_point}, the subspace $V$ here does not depend on the scale where the doubling index is pinched and the points in $\cV^u_{\epsilon,d,r}$.

Finally, we prove the generalization of Lemma \ref{l:Harmonic_no_critical_pt_away_from_n-2}.
\begin{lemma}\label{l:Ellpitic_no_critical_pt_away_from_n-2}
    Let $u: B_2 \to \RR$ be a solution to (\ref{e:elliptic_equ}) (\ref{e:assumption}) with doubling assumption \ref{e:DI_bound_assumption}. Let $\tau < c_3(n,\lambda)$ and $x \in B_1$. If $\epsilon \leq \epsilon_0= c(n, \alpha,\tau)^{\Lambda}$ and $r \leq r_0 =   C(n,\lambda, \alpha)^{-\Lambda}\epsilon^{1/\alpha}$, then the following holds. If $u$ is $(n-2, \sqrt{\epsilon}, 3r, x)$-symmetric with respect to $V$, then $C(u) \cap B_r(x) \subset B_{\tau r}(x + V)$.
\end{lemma}

\begin{proof}
    According to (the proof of) harmonic approximation proposition \ref{p:harmonic_appro}, the approximated harmonic function $h$ is actually $C^1$ close to $u$. Then the Lemma follows from Lemma \ref{l:Harmonic_no_critical_pt_away_from_n-2} and Proposition \ref{p:Cone_Split_Elliptic}.
\end{proof}

\section{Volume Estimates on Critical Sets}
In this section, following \cite{NV}, we will first prove a covering Lemma and then finish the proof of our main theorem \ref{t:critical_set}.

\subsection{Covering Lemma}

In this subsection, we will prove the covering lemma. Throughout this section, we assume $u$ is a solution to \ref{e:elliptic_equ}, \ref{e:assumption} with doubling assumption \ref{e:DI_bound_assumption} and that $C(u)$ is the critical point of $u$. 

In this section. we pick $\tau \leq c_3(n,\lambda)$ and $\epsilon = c_4(n,\alpha,\tau)^{-\Lambda}$ and as in Lemma \ref{c:Elliptic_Pinch_pt_near_n-2}. Then choose $r \leq  r_0 = (C_1(n,\lambda,\alpha))^{-\Lambda} \epsilon^{1/\alpha}$ as in Lemma \ref{l:DI_Drop_Elliptic}. We will always work in the ball $B_r(x)$ with $r \leq r_0$. However, for simplicity we will rescale the ball to the unit scale and this will not affect our final estimates.

\begin{definition} 
    A ball $B_r(x)$ is called a $(d,t)$-good ball if for any $y \in B_r(x)$, we have $D(y,s) \leq d + \epsilon$ for any $s \leq tr$.
\end{definition}

\begin{lemma}\label{l:Cover_Good}
    Let $r>0$ be fixed and $B_1(0)$ be a $(d,C_2(\lambda)\Lambda)$-good ball with $d\geq 2$. Then the following covering lemma holds 
    \begin{equation}
        C(u) \cap B_1 \subset \bigcup_i B_{s_i}(x_i) \quad \text{ and } \quad \sum_i s_i^{n-2} \leq C_4(n,\lambda) \Lambda^n \epsilon^{-C_5(n, \alpha)}. 
    \end{equation}
    where for each $B_{s_i}(x_i)$ either $s_i \leq r$ or it is a $(d-1, C_2(\lambda)\Lambda)$-good ball. Here $C_2(\lambda) = 10^{3}(1+\lambda)^{1/2}$. 
\end{lemma}

\begin{proof}
We follow the proof in \cite{NV}. Without loss of generality, we may assume $r \leq (C_2(\lambda)\Lambda)^{-1}$. For each point $x \in B_1(0)$, we define 
    \begin{equation}
        r'_x \equiv \sup\{ 0 \leq s \leq 1 : D(x,s) \leq d - \epsilon \}.
    \end{equation}
    If no such $r'_x$ exists, we set $r'_x = 0$. Now we define 
    \begin{equation}
        r_x \equiv \max\{r'_x, r\}.
    \end{equation}

Further we decompose the critical set in $B_1$ as 
\begin{equation}
    C_a(u) \equiv \{ x \in C \cap B_1: r_y \geq r_x/7 \text{ for any } y \in B_{5r_x}(x) \}.
\end{equation}
\begin{equation}
    C_b(u) \equiv \{ x \in C \cap B_1: r_y < r_x/7 \text{ for some } y \in B_{5r_x}(x) \}.
\end{equation}

\textbf{Step 1: Covering of $C_a.$}

We pick a covering of $C_a$ such that each $x_i \in C_a$ and
\begin{equation}
    C_a \subset \bigcup_i B_{5r_i}(x_i) \quad \text{ and } \quad B_{r_i}(x_i) \cap B_{r_j}(x_j) = \emptyset \text{ if } i \neq j.
\end{equation}
where $r_i = r_{x_i}$. By the definition of the ball $C_a$, for any $y \in B_{r_i}(x_i)$, we have $r_y \geq r_x/7$. If $r_i>r$, By Lemma \ref{l:DI_Drop_Elliptic} we have
\begin{equation}
    D_{y}(\epsilon^{c(\alpha)}r_i/7) \leq d - 1 +\epsilon \text{ for any } y \in B_{r_i}(x_i).
\end{equation}
This prove that each ball $B_{r_i}(x_i)$ is a $(d-1, \epsilon^{c(\alpha)}/7)$-good ball if $r_i>r$. For each ball $B_{r_i}(x_i)$, we can apply the Vitali covering again using the balls $B_{i,j}$ with radius $7 \epsilon^{c(\alpha)}(C_2(\lambda)\Lambda)^{-1}r_i$. Then each ball $B_{i,j}$ is a $(d-1, C_2(\lambda)\Lambda)$-good ball with $d \geq 2$ and the number of these balls is bounded by $C(n, \lambda)\Lambda^n \epsilon^{-C(n, \alpha)}$. The proof for 
 the covering of $C_a$ will be finished once we can prove that $\sum_i r_i^{n-2} \leq C(n)$. 

Next we prove the measure estimate $\sum_i r_i^{n-2} \leq C(n)$. By volume estimate and the disjointness assumption, we can control the number of those balls with $r_i \geq 1/100$ and thus obtain the bound
\begin{equation}
    \sum_{r_i \geq 1/100} r_i^{n-2} \leq C(n). 
\end{equation}

Hence in the following we only consider those points where $r_i \leq 1/100$. We partition those points further into the groups $G_{\alpha \in \cA}$ where for any $x, y \in G_{\alpha}$ we have $|y-x| \leq 1/100$. Easy to see the cardinality of $\cA$ is bounded by some constant $C^n$.

Now take any $x, y \in G_{\alpha}$. Let $R = |x-y| \geq r_x + r_y$. Then $y  \in \cV_{d, \epsilon, R}(x)$. Then by Corollary \ref{c:Elliptic_Pinch_pt_near_n-2}, there exists some $n-2$-dimensional subspace $V$ such that if $\tau \leq c_3(n,\lambda) \leq 1/100$ then
\begin{equation}\label{e:Lipschitz}
    d(y-x, V) \leq \tau R \leq |y-x|/100.
\end{equation}

By Corollary \ref{c:Elliptic_Pinch_pt_near_n-2}, the subspace $V$ can be chosen independent of $x$ and $y$. Hence the equation \ref{e:Lipschitz} holds for any $x, y \in G_{\alpha}$ with respect to the same $V$. By Lipschitz extension theorem, there exists a Lipschitz map $f: V \to V^{\perp}$ such that Lip$(f) \leq 1/10$ and $\cV_{d, \epsilon, 1} \subset $ graph$(f)$. Therefore, by the Lipschitz construction and disjointness, we have
\begin{equation}
    \sum_{x_i \in G_{\alpha}} r_i^{n-2} \leq C(n). 
\end{equation}

Summing over all subfamilies $G_{\alpha}$ for $\alpha \in \cA$, we have
\begin{equation}
    \sum_{r_i \leq 1/100} r_i^{n-2} \leq C(n). 
\end{equation}

\textbf{Step 2: Covering of $C_b$.}

Let $y \in C_b$. By definition, there exists some $x_0 \in C \cap B_{5r_y}$ such that $r_{x_0} < r_y/7$. If $x_0 \in C_a$, then we stop. Otherwise, $x_0 \in C_b$ and we can find $x_1 \in C \cap B_{5 r_{x_0}}$ such that $r_{x_1} \leq r_{x_0}/7$. Iterating this process, one can eventually find some $x \in C_a$ such that $r_x \leq r_y/7$ and
\begin{equation}
    |x-y| \leq |y - x_0| + |x_0 - x_1| + ... \leq (\sum_{k=0}^{\infty} 7^{-k}) 5r_y \leq 6r_y.
\end{equation}

According to Step 1, there exists some $B_{5r_i}(x_i)$ in the covering of $C_a$ such that $x\in B_{5x_i}$. Since $r_x \geq r_i/7$, we have $r_i \leq r_y$. Hence, 
\begin{equation}
    |y - x_i| \leq |y-x| + |x - x_i| \leq 6r_y + 5r_i \leq 11 r_y. 
\end{equation}

Now for each $y \in C_b$, we set 
\begin{equation}
    t_y \equiv \frac{1}{55}\min_i |y -x_i| \leq r_y/5,
\end{equation}
where the minimal is taken over all $x_i$ in Step 1.
We can cover the set $C_b \setminus \bigcup_i B_{5r_i}(x_i)$ by $\bigcup_{y \in C_b} B_{t_y}(y)$ and choose a Vitali subcovering, i.e.
\begin{equation}
    C_b \setminus \big(\bigcup_i B_{5r_i}(x_i) \big) \subset \bigcup_j B_{5t_j}(y_j)  \quad \text{ and } \quad B_{t_j}(y_j) \cap  B_{t_k}(y_k) = \emptyset \text{ if } j \neq k. 
\end{equation}

First we verify the doubling index drop condition. Take any $z \in C \cap B_{5t_j}(y_j)$. We want to prove $r_z \geq t_j$. 

If $z \in C_a \cap B_{5t_j}(y_j)$, then $z\in B_{5r_i}(x_i)$ for some $i$. Hence
\begin{equation}
    55t_j \leq |y_j - x_i| \leq |y_j - z| + |z - x_i| \leq 5t_j + 5r_i. 
\end{equation}

Therefore $r_i \geq 10t_j$. Since $z\in B_{5r_i}(x_i)$, we have $r_z \geq r_i/7 \geq 10t_j/7$. 

If $z \in C_b \cap B_{5t_j}(y_j)$, then there exists some $x_i$ such that
\begin{equation}
    r_z \geq \frac{1}{11}|z - x_i| \geq \frac{1}{11}(|y_j - x_i| - |z - y_j|) \geq \frac{1}{11}( 55t_j - 5t_j) = \frac{50}{11}t_j.
\end{equation}

Hence in any case, $r_z \geq t_j$. According to Lemma \ref{l:DI_Drop_Elliptic}, we have
\begin{equation}
    D(z,\epsilon^{c(\alpha)}t_j) \leq d-1-\epsilon.
\end{equation}

Similarly, the proof will be finished once we prove the following measure bound
\begin{equation}
    \sum_{j \in J} t_j^{n-2} \leq C(n).
\end{equation}

Recall that in Step 1 we obtain a Lipschitz function $f$ defined on a subspace $V$ of dimension at most $n-2$ such that each $x_i$ lies on the Graph$(f)$. We first consider the case dim$V < n-2$. 

We partition the index set $J = \cup_{k=0}^{\infty} J_k$, where $j \in J_k$ if $t_j \in (2^{-k-1}, 2^{-k}]$. By the definition of $t_j$, we have $d(y_j, \text{Graph}(f)) \leq 55t_j$. If $j \in J_k$, then 
\begin{equation}
    B_{t_j}(y_j) \subset B_{56\cdot 2^{-k}}(\text{Graph}(f)).
\end{equation}
By the disjointness and comparing the volume, we can control the cardinality of $J_k$, i.e.
\begin{equation}
    |J_k| \leq \frac{\Vol(B_{56\cdot 2^{-k}}(\text{Graph}(f)))}{\Vol(B_{2^{-k-1}})} \leq C(n) 2^{k(n-3)}.
\end{equation}

Therefore, summing over all $k$ we will get
\begin{equation}
    \sum_{j \in J} t_j^{n-2} \leq \sum_k |J_k| 2^{-k(n-2)} \leq C(n).
\end{equation}

It remains to consider the case where dim$(V) = n -2$. Then by Lemma \ref{l:Ellpitic_no_critical_pt_away_from_n-2}, for any $x_i$, there is no critical point of $u$ in this set
\begin{equation}
    \{ z \in B_1(x_i) \setminus B_{5 r_i}(x_i) : d(z - x_i, V) \geq \tau |z - x_i| \}. 
\end{equation}

Recall that $55 t_j = \min_i |y_j -x_i|$. Then $ 55 t_j = |y_j - x_{j'}|$ for some $j'$. Since $y_j \in C(u) \setminus B_{5r_{j'}}(x_{j'})$, we have $d(y_j -x_{j'}, V) \leq \tau |y_j - x_{j'}| \leq t_j/100$ as $\tau \leq 10^{-4}$. Similarly, for any other $y_{\ell}$, we have $x_{\ell'}$ such that 
$d(y_\ell -x_{\ell'}, V) \leq \tau |y_\ell - x_{\ell'}| \leq t_\ell/100$. Noting that 
\begin{align}
  d(x_{j'},x_{\ell'})\le d(x_{j'},y_{j})+d(y_{\ell},x_{\ell'}) +d(y_{j},y_{\ell})\le (55t_j+55t_\ell)+ d(y_{j},y_{\ell})\le 56d(y_{j},y_{\ell})
\end{align} Thus 
\begin{align}
d(y_\ell -y_{j}, V)&\le d(y_\ell -x_{\ell'}, V)+d(y_j -x_{j'}, V)+d(x_{j'} -x_{\ell'}, V)\\
&\le t_\ell/100+t_j/100+56d(y_j,y_\ell)\tau \\
&\le (56\tau+1/100)d(y_j,y_\ell)\le 1/50d(y_j,y_\ell).
    \end{align}
By the disjointness, we have
\begin{equation}
    \sum_j t^{n-2}_j \leq C(n). 
\end{equation}

This finishes the proof.
\end{proof}

\subsection{Proof of Main Theorem}

In this subsection, we apply the previous covering Lemma to prove the volume estimates of critical sets and the singular sets.

\begin{proof}[Proof of Theorem \ref{t:critical_set}]
   Pick $\tau \leq c_3(n,\lambda)$ and $\epsilon = c_4(n,\alpha,\tau)^{-\Lambda}$ and as in Lemma \ref{c:Elliptic_Pinch_pt_near_n-2}. Then choose $r \leq  r_0 = (C_1(n,\lambda,\alpha)\Lambda)^{-\Lambda} \epsilon^{1/\alpha}$ as in Lemma \ref{l:DI_Drop_Elliptic}. For $r\geq r_0$, by choosing sufficiently large $C$ the estimate holds.

    First we cover the unit ball $B_1$ using the balls $B_{r_{i,0}}(x_{i,0})$ with radius $r_{i,0}=r_0$. By Lemma \ref{l:Bound_on_General_DI}, each ball is a $(C_0(n,\lambda) \Lambda, C_2(\lambda)\Lambda)$-good ball. And we can control the number of the balls in this covering by $(C_5(n,\lambda,\alpha)\Lambda)^{\Lambda}$.

      Let $d$ be the smallest integer satisfying $C(n,\lambda)\Lambda \leq d_1$. Hence each ball $B_{r_{i,0}}(x_{i,0})$ is a $(d,  C_2(\lambda)\Lambda))$-good ball. According the Covering Lemma \ref{l:Cover_Good} to each ball, we have
    \begin{equation}
        C(u) \cap B_1 \subset \bigcup_i B_{r_{1, i}}(x_{1,i})\quad \text{ and } \quad \sum_i r_{1, i}^{n-2} \leq (C_5(n,\lambda,\alpha)\Lambda)^{\Lambda} C_6(n,\lambda,\alpha)^{\Lambda},
    \end{equation}
where each ball $B_{s_{1, i}}(x_{1,i})$ is a $(d-1, C_2(\lambda)\Lambda)$-good ball. 

Now we can repeat the previous covering construction to each ball $B_{s_{1, i}}(x_{1,i})$ to refine the covering
    \begin{equation}
        C_(u) \cap B_1 \bigcup_i B_{r_{2, i}}(x_{2,i})\quad \text{ and } \quad \sum_i r_{2, i}^{n-2} \leq (C_5(n,\lambda,\alpha)\Lambda)^{\Lambda} C_6(n,\lambda,\alpha)^{2 \Lambda},
    \end{equation}
where each ball $B_{s_{2, i}}(x_{1,i})$ is a $(d-2, C_2(\lambda)\Lambda)$-good ball. 
    
Iterate this process for at most $C(n,\lambda)\Lambda$ times and we will end up with a covering
    \begin{equation}
        C(u) \cap B_1 \subset \bigcup_i B_{r_{i}}(x_{i})\quad \text{ and } \quad \sum_i r_{ i}^{n-2} \leq (C_5(n,\lambda,\alpha)\Lambda)^{\Lambda} C_6(n,\lambda,\alpha)^{C_0(n,\lambda)\Lambda^2} \leq C_7(n,\lambda,\alpha)^{\Lambda^2}.
    \end{equation}

Note that eventually either $r_i = r$ or $B_{r_i}(x_i)$ is a $(1, C_2(\lambda)\Lambda)$-good ball. According to the $\epsilon$-regularity proposition \ref{p:Elliptic_gradient_lower_bound}, $C(u) \cap B_{r_i}(x_i) = \emptyset$. Therefore, at this final step, the only possibility is $r_i = r$. We prove that the number of these balls is bounded by $C_7(n,\lambda,\alpha)^{\Lambda^2} r^{2-n}$. This proves the Minkowski estimate.
\end{proof}

\begin{proof}[Proof of Theorem \ref{t:singular_set}]
In this case, compared to Definition \ref{d:doubling_index}, we define the doubling index as 
    \begin{equation}
        {D}^u(x,r) \equiv \text{log}_4 \frac{\fint_{ \partial B_{2}(0)} u(x +rA_x(y))^2 dy}{\fint_{ \partial B_1(0)} u(x +rA_x(y))^2 dy}.
    \end{equation}

With this one can prove the almost monotonicity theorem, quantitative uniqueness of tangent maps and cone splitting Lemma by the same arguments. Note that in this case $D(x, r) \to 0$ as $r \to 0$ when $u(x) \neq 0$. Hence the $\epsilon$-regularity like Proposition \ref{p:Elliptic_gradient_lower_bound} holds only at the points where $u(x)=0$ as $\liminf D(x, r) \geq 1$ with $u(x) = 0$. By decomposing the singular set, the theorem follows from the same covering argument.
\end{proof}

\section{Neck Region Construction}

In this chapter, we will prove the sharp estimate for each quantitative stratum as in \cite{CNV} using Neck region decomposition. First we introduce the notion of neck regions which was first proposed in \cite{JN} \cite{NVYM}. Recall that $C_2(\lambda) = 10^3(1+\lambda)^{1/2}$. We will let $\tau \leq c_3(n,\lambda)$ be a small fixed constant. 

\begin{definition}{(Neck Region)}
    Let $u: B_2 \to \RR$ be a solution to (\ref{e:elliptic_equ}) (\ref{e:assumption}) with doubling assumption \ref{e:DI_bound_assumption}. Let $\cC \subset B_{r}$ a closed subset and $r_x: \cC \to \RR^+$ a radius function such that the closed balls \{$\bar{B}_{ r_x/5}(x)$\} are disjoint. Let $d$ be an integer. The subset $\cN = B_r \setminus \Bar{B}_{r_x}(\cC)$ is called a $(d, k,  \epsilon, \eta)$-neck region if
\begin{enumerate}
    \item For any $x\in \cC$, we have $|D(x, s) - d| \leq \epsilon$ for any $s \in [r_x, C_2(\lambda)r]$ with $r_x \leq 10^{-3}r$.
    \item For any $x\in \cC$, $u$ is uniformly $(k,\epsilon,x)$-symmetric in $[r_x, 10r]$ with respect to $V_x$ but $u$ is not $(k+1,\eta,s,x)$-symmetric for any $s \in [r_x, 10r]$.  
    \item For any $B_{2s}(x) \subset B_{2r}$ with $s \geq r_x$, we have $\cC \cap B_{s}(x) \subset B_{\tau s}(x + V_x)$. 
\end{enumerate}
\end{definition}

Unlike the cases in \cite{CJN} \cite{JN} \cite{NVYM}, here the approximated subspace $V_x$ does not depend on the scale $r$ according to the result of the uniqueness of tangent maps and the corresponding cone-splitting Lemma. As one will see, this can greatly simplify the decomposition arguments. 

\subsection{Neck Structure Theorem}

In this subsection, we will prove the following neck structure theorem. The Lipschitz structure is crucial in the neck region construction but its proof is usually much more involved as in \cite{CJN} \cite{JN} \cite{NVYM}. However, the invariance property of scale for $V_x$ in our setting makes this step much easier. 

\begin{theorem}{(Neck Structure Theorem)}\label{t:Neck_Struture}
    If $\cN = B_{2r} \setminus \bar{B}_{r_x}(\cC)$ is a $(d, k, \epsilon, \eta)$-neck region, then 
    \begin{enumerate}
        \item There exists a Lipschitz function $f: V \to V^{\perp}$ such that $\cC \subset \text{Graph}(f)$ with Lip$(f) \leq 10^{-1}$. Here $V$ can be chosen as $V_x$ for any $ x\in \cC$. 
        \item We have the measure bound $\sum_{x\in \cC_+} r_x^{k} \delta_x + \cH^{k}|_{\cC_0} \leq C(n) r^k$.
    \end{enumerate}
\end{theorem}

\begin{proof}
    We define the distance between two affine planes $V, W$ as
    \begin{equation}
        d(V, W) = d^{H}(\Tilde{V} \cap B_1, \Tilde{W} \cap B_1),
    \end{equation}
where $\Tilde{V}, \Tilde{W}$ are the subspaces parallel to $V$ and $W$, and $d^H$ is the Hausdorff distance. Let $\pi_V$ denote the projection on $V$. It is easy to check that 
\begin{equation}
    ||\pi_V(x) - \pi_W(x)|| \leq 2 d(V, W) ||x||. 
\end{equation}

\textbf{Claim: } $d(V_x, V_y) \leq 10\tau$ for any $x, y \in \cC$.
    
\textit{Proof of Claim: } Consider $x, y \in \cC$. Suppose $r_x \geq r_y$. If $||x - y || \geq r_x$, then according to item (3) we have $d(V_x, V_y) \leq 4\tau$. If $||x - y || \leq r_x$, then $y \in B_{\tau r_x}(V_x) \setminus B_{r_x/5}(x)$. Then by item (3) again we have $d(V_x, V_y) \leq 10\tau$. This proves the claim.

Now we fix $w \in \cC$ and let $W \equiv V_w$. Let $x, y \in \cC$. Since $B_{r_x/5}(x)$ and $B_{r_y/5}(y)$ are disjoint, then $10 ||x-y|| \geq r_x + r_y$. Hence we can apply the item (3) to the scale $s = 10 ||x-y|| $ and then
    \begin{equation}
        ||\pi^{\perp}_{V_x}(x) - \pi^{\perp}_{V_x}(y)|| \leq \tau s = 10 \tau  ||x-y||
    \end{equation}
    where $\pi^{\perp}_{V_x}$ means the orthogonal projection on ${V_x}^{\perp}$. By triangle inequality, we have
    \begin{equation}
        ||\pi_{V_x}(x) - \pi_{V_x}(y)|| \geq ||x - y|| - ||\pi^{\perp}_{V_x}(x) - \pi^{\perp}_{V_x}(y)|| \geq (1 - 10 \tau) ||x-y||. 
    \end{equation}

Moreover, according to the claim, 
\begin{equation}
    ||\pi_W(x) - \pi_W(y)|| \geq ||\pi_{V_x}(x) - \pi_{V_x}(y)|| - 2d({V_x}, W) ||x- y|| \geq (1 - 30 \tau) ||x - y||. 
\end{equation}

Since $\tau \leq 10^{-3}$, there exists a Lipshitz map $f: W \to W^{\perp}$ with Lip$(f) \leq 10^{-1}$ such that $\cC \subset $graph$(f)$ by Lipschitz extension theorem. The measure bound comes directly from the Lipschitz construction.
\end{proof}

\subsection{Neck Decomposition Theorem}
In this section, we will prove the following Neck decomposition theorem. 

\begin{theorem}{(Neck Decomposition Theorem)} \label{t:Neck_decompose}
    Let $u: B_{2} \to \RR$ be a solution to (\ref{e:elliptic_equ}) (\ref{e:assumption}) with doubling assumption \ref{e:DI_bound_assumption}. For each $\eta>0$ and $\epsilon \leq \epsilon(n, \lambda, \alpha, \Lambda, \eta)$ we have
\begin{equation}
\begin{split}
    &B_1 \subset \bigcup_a ( \cN^a \cap B_{r_a}(x_a) \big) \cup \bigcup_b B_{r_b}(x_b) \cup S^{k, \epsilon, \eta}  \\
    &  S^{k, \epsilon, \eta} \subset S_0 \cup \bigcup_{a} \cC_{0, a} 
\end{split}
\end{equation}
with the estimates
\begin{equation}
    \sum_a r^k_a + \sum_b r^k_b + \cH^k(S^{k, \epsilon, \eta}) \leq C(n, \lambda, \alpha, \Lambda, \epsilon, \eta).
\end{equation}
where $S^{k, \epsilon, \eta}$ is $k$-rectifiable and $\cH^k(S_0)=0$. Moreover, for any $\delta>0$, if $\eta < \eta(n, \lambda, \alpha, \Lambda, \delta)$ and $\epsilon < \epsilon(n, \lambda, \alpha, \Lambda, \eta, \delta)$, the quantitative stratum $\cS^k_{\delta}$ satisfies
\begin{equation}
    \cS^k_{\delta} \subset S^{k, \epsilon, \eta}.
\end{equation}
\end{theorem}

The proof of Theorem \ref{t:Estimate_on_Strata} follows from Theorem \ref{t:Neck_decompose} and cone splitting Lemma. We will focus on the proof of Theorem \ref{t:Neck_decompose} and invite the interested readers to consult \cite{CJN} for more details.

First we recall the cone-splitting Lemma for elliptic solutions. This should be compared to Proposition \ref{p:Cone_Split_Elliptic}. 

\begin{lemma}{(Cone-Splitting)}\label{l:Neck_cone_splitting}
    Let $u: B_2 \to \RR$ be a solution to (\ref{e:elliptic_equ}) (\ref{e:assumption}) with doubling assumption \ref{e:DI_bound_assumption}. Let $\eta > 0$. There exists some $\epsilon_0(n,\lambda,\alpha,\Lambda, \eta) > 0$ and $r_0(n,\lambda,\alpha,\Lambda, \eta) > 0$ such that the following is true. Let $r_1 \leq r_0$ and $\epsilon \leq \epsilon_0$. Let $x\in B_1$ and $y \in B_{r_1}(x) \setminus B_{\tau}(x)$. Suppose $|D(x, r) - D(x, r_1)| \leq \epsilon$ and $|D(y, r) - D(y, r_1)| \leq \epsilon$ for any $r \in [10^{-3}r_1, C_2(\lambda)r_1]$. Then we have $u$ is uniformly $(1, \eta, x)$-symmetric in $[10^{-2}r_1, 10 r_1]$. 
\end{lemma}

\begin{proof}
    The proof follows from the standard contradiction argument and the uniqueness of tangent maps Theorem \ref{t:Quan_Unique}. 
\end{proof}

The splitting Lemma can be easily generalized to higher dimensional case.

Now we define 
\begin{equation}
    \cV^u_{\epsilon, d, r} (x) \equiv \{ y \in B_{r}(x) : |D(y,s) - d| \leq \epsilon \text { for any } s \in [ 10^{-3}r , C_2(\lambda) r] \},
\end{equation}

We will omit the superscript in the following.  

Throughout this section, we will use different subscripts to denote the balls with different structures. We list all the categories here:
\begin{enumerate}
    \item [(a)] A ball $B_{r_a}(x_a)$ is associated with a $(d, k,\epsilon,\eta)$-neck region $\cN_a = B_{2r_a}(x_a) \setminus \bar{B}_{r_{a,x}}(\cC_a)$. 
    \item [(b)] A ball $B_{r_b}(x_b)$ is such that $u$ is $(k+1, 2\eta, 2r_b, x_b)$-symmetric.
    \item [(c)] A ball $B_{r_c}(x_c)$ is not a $b$-ball and $\cV_{\epsilon,d,r_c}$ is $(k, \tau)$-independent in $B_{r_c}(x_c)$.
    \item [(d)] A ball $B_{r_d}(x_d)$ is such that $\cV_{\epsilon,d,r_d} \neq \emptyset$ and is not $(k, \tau)$-independent in $B_{r_d}(x_d)$.
    \item [(e)] A ball $B_{r_e}(x_e)$ is such that $\cV_{\epsilon,d,r_e} = \emptyset$.
\end{enumerate}

We will prove covering Lemmas for those balls in the following. First we prove a covering of $d$-balls.

\begin{proposition}[Covering of $d$-balls]
 Let $u: B_{2} \to \RR$ be a solution to (\ref{e:elliptic_equ}) (\ref{e:assumption}) with doubling assumption \ref{e:DI_bound_assumption}.  Let $\epsilon \leq \epsilon_0= C(n, \lambda, \alpha, \Lambda)$ and $r \leq r_0 =   C(n,\lambda, \alpha, \Lambda, \epsilon)$. Suppose $\sup_{x \in B_r} D(x,s) \leq d+\epsilon$ for any $s \leq C_2(\lambda) r$ and for some integer $d$. Furthermore we assume that $\cV_{\epsilon,d,r}(0) \neq \emptyset$ and is not $(k,\tau)$-independent in $B_r$. Then we have the following decomposition
 \begin{equation}
     B_r \subset \bigcup_{b \in B} B_{r_b}(x_{b}) \cup \bigcup_{c \in C} B_{r_c}(x_{c}) \cup \bigcup_{e \in E} B_{r_e}(x_{e}) \cup S_d,
 \end{equation}
where each $B_{r_b}$ is a $b$-ball, $B_{r_c}(x_{c})$ is a $c$-ball and $B_{r_e}(x_{e})$ is an $e$-ball. Furthermore, we have the following measure estimates
\begin{equation}
\begin{split}
    & \sum_{b \in B} r_b^k + \sum_{e \in E} r_e^k \leq C(n,\tau) r^k. \\
    & \sum_{c \in C^1} r_c^k \leq C(n) \tau r^k. \\
    & \cH^k(S_d) = 0.
\end{split}
\end{equation}
    
\end{proposition}

\begin{proof}

First we choose a Vitali covering of $B_r$, i.e. 
\begin{equation}
    B_r \subset \bigcup_{f \in F^1} B_{10^{-1} \tau r}(x^1_{f}) \quad \text{ and } \quad B_{10^{-2}\tau r}(x^1_{f_i}) \cap B_{10^{-2}\tau r}(x^1_{f_j}) = \emptyset \text{ if } i \neq j.
\end{equation}

Since each $B_{10^{-1}r}(x^1_{f})$ must be of one of the $b$, $c$, $d$ or $e$-type, we have
\begin{equation}
     B_r \subset \bigcup_{b \in B^1} B_{10^{-1}\tau r}(x^1_{b}) \cup \bigcup_{c \in C^1} B_{10^{-1}\tau r}(x^1_{c}) \cup \bigcup_{d \in D^1} B_{10^{-1}\tau r}(x^1_{d}) \cup \bigcup_{e \in E^1} B_{10^{-1}\tau r}(x^1_{e}).
\end{equation}

By the disjointness, the number of these balls is bounded by $C(n)\tau^{-n}$. Hence we have $\sum_{f \in F^1} r^k_{f} \leq C_5(n, \tau) r^k$. Note that $C_5(n,\tau) = C(n) \tau^{k-n}$ .

Since $\cV_{\epsilon,d,r}(0) \neq \emptyset$ and is not $(k,\tau)$-independent in $B_r$, there exists some $(n-1)$-dimensional plane $V^{k-1}$ such that $\cV_{\epsilon,d,r}(0) \in B_{\tau}(V)$. As $\cV_{\epsilon,d,r}(0) \neq \emptyset$ in $c$-balls and $d$-balls, we have 
\begin{equation}
    \bigcup_{c \in C^1} B_{10^{-1}\tau r}(x^1_{c}) \cup \bigcup_{d \in D^1} B_{10^{-1}\tau r}(x^1_{d}) \subset B_{2\tau r}(V).
\end{equation}

Note that vol$(B_{2\tau r}(V)) = C(n)\tau^{n-k+1} r^n$. By the disjointness property of these balls, the number of $c$-balls and $d$-balls is bounded by $C(n)\tau^{-k+1}$. Therefore we have
\begin{equation}
    \sum_{c \in C^1} r_c^k + \sum_{d \in D^1} r_d^k \leq C_6(n) \tau r^k.
\end{equation}

In conclusion, we have the covering of $B_r$ satisfying
\begin{equation}
\begin{split}
    & B_r \subset \bigcup_{b \in B^1} B_{10^{-1}\tau r}(x^1_{b}) \cup \bigcup_{c \in C^1} B_{10^{-1}\tau r}(x^1_{c}) \cup \bigcup_{d \in D^1} B_{10^{-1}\tau r}(x^1_{d}) \cup \bigcup_{e \in E^1} B_{10^{-1}\tau r}(x^1_{e}). \\
    & \sum_{b \in B^1} r_b^k + \sum_{e \in E^1} r_e^k \leq C_5(n,\tau) r^k. \\
    & \sum_{c \in C^1} r_c^k + \sum_{d \in D^1} r_d^k \leq C_6(n) \tau r^k.
\end{split}
\end{equation}

Now for each $d$-ball $B_{r_d}(x^1_d)$, we repeat this process such that the covering of $B_r$ is refined to be
\begin{equation}
\begin{split}
    & B_r \subset \bigcup_{b \in B^2} B_{r_b}(x^2_{b}) \cup \bigcup_{c \in C^2} B_{r_c}(x^2_{c}) \cup \bigcup_{d \in D^2} B_{10^{-2}\tau^2 r}(x^2_{d}) \cup \bigcup_{e \in E^2} B_{r_e}(x^2_{e}). \\
    & \sum_{b \in B^2} r_b^k + \sum_{e \in E^2} r_e^k \leq C_5(n,\tau) (1 + C_6(n) \tau) r^k. \\
    & \sum_{c \in C^2} r_c^k  \leq C_6(n) \tau (1 + C_6(n) \tau) r^k. \\
    & \sum_{d \in D^2} r_d^k = \sum_{d \in D^2} (\frac{\tau}{10})^{2k} r^k \leq (C_6(n) \tau)^2 r^k.
\end{split}
\end{equation}

After iteration for $N$-times, we have the following covering for $B_r$
\begin{equation}
\begin{split}
    & B_r \subset \bigcup_{b \in B^N} B_{r_b}(x^N_{b}) \cup \bigcup_{c \in C^N} B_{r_c}(x^N_{c}) \cup \bigcup_{d \in D^N} B_{(10^{-1}\tau)^N r}(x^N_{d}) \cup \bigcup_{e \in E^N} B_{r_e}(x^N_{e}). \\
    & \sum_{b \in B^N} r_b^k + \sum_{e \in E^N} r_e^k \leq C_5(n,\tau) \sum_{i = 0}^{N-1} (C_6(n) \tau)^i r^k. \\
    & \sum_{c \in C^N} r_c^k  \leq C_6(n) \tau \sum_{i = 0}^{N-1} (C_6(n) \tau)^i r^k. \\
    & \sum_{d \in D^N} (\frac{\tau}{10})^{kN} \leq (C_6(n) \tau)^N.
\end{split}
\end{equation}

Let $S^i \equiv \{ x^i_d \}_{d\in D^i}$. Note that $S^{i+1} \subset B_{(10^{-1}\tau)^i r}(S^i)$. We can define the Hausdorff limit of $S^i$ to be $S_d$. Then for any $i \geq 1$
\begin{equation}
    \text{Vol}(B_{(10^{-1}\tau)^i r}(S_d)) \leq \text{Vol}(B_{(10^{-1}\tau)^i r}(S^i)) \leq \sum_{d \in D^i} \text{Vol}(B_{(10^{-1}\tau)^i r}(x^i_d)) \leq C(n) (C_6(n)\tau)^i (\frac{\tau}{10})^{i(n-k)} r^n.
\end{equation}

Therefore, the proof is finished by choosing $C_6(n) \tau < 1/2$.
\end{proof}

Next we prove the covering of $c$-balls. We will construct a neck region in each $c$-ball.

\begin{proposition}{(Covering of $c$-balls)}
    Let $u: B_{2} \to \RR$ be a solution to (\ref{e:elliptic_equ}) (\ref{e:assumption}) with doubling assumption \ref{e:DI_bound_assumption}. Let $\eta > 0$. Suppose $\epsilon \leq \epsilon_0= C(n, \lambda, \alpha, \Lambda, \eta)$ and $r \leq r_0 =   C(n,\lambda, \alpha, \Lambda, \epsilon, \eta)$. Suppose $\sup_{x \in B_r} D(x,s) \leq 
    d+\epsilon$ for any $s \leq C_2(\lambda) r$ and for some integer $d$. Furthermore we assume that $\cV_{\epsilon,d,r}(0)$ is in $(k,\tau)$-position and that $B_2$ is not $(k+1, 2\eta)$-symmetric. Then we can decompose
    \begin{equation}
     B_r \subset  \big(\cC_0 \cup \cN)  \cup \bigcup_{b \in B} B_{r_b}(x_{b}) \cup \bigcup_{d \in D} B_{r_d}(x_{d}) \cup \bigcup_{e \in E} B_{r_e}(x_{e}),
\end{equation}
where $\cN$ is a $(d, k, \epsilon, \eta)$-neck region. Furthermore, we have the estimates
\begin{equation}
    \sum_{b \in B} r_b^k + \sum_{d \in D} r_d^k + \sum_{b \in E} r_e^k + \cH^k(\cC_0) \leq C(n) r^k.
\end{equation}
\end{proposition}

\begin{proof}

For each point $x \in B_r$, we define 
    \begin{equation}
        r_x \equiv \sup\{ 0 \leq s \leq C_2(\lambda) r : D(x,s) \leq d - \epsilon \}.
    \end{equation}
    If no such $r_x$ exists, we set $r_x = 0$. Note that $x \in \cV_{\epsilon,d,r} $ if $r_x \leq 10^{-3} r$. Define $\cC_0 \equiv \{ x \in B_r : r_x = 0 .\} $ We can choose a Vitali subcovering of $\bigcup_{x \in \cV_{\epsilon,d,r}, r_x > 0} B_{r_x}(x)$, i.e.
    \begin{equation}
        \bigcup_{x \in \cV_{\epsilon,d,r}, r_x > 0} B_{r_x}(x) \subset \bigcup_{x \in \cC_+} B_{r_x}(x), \quad B_{r_x/5}(x) \cap B_{r_y/5}(y) = \emptyset \text{ if } x \neq y.
    \end{equation}
    
By Cone-splitting Lemma, there exists some affine $k$-plane $V$ such that $\cV_{\epsilon,d,r}(x) \subset B_{\tau r}(V)$. 

We define the first neck region by $\cN^1 \equiv B_{2r} \setminus \bigg( \cC_0 \cup \bigcup_{x \in \cC_+ }\Bar{B}_{r_x}(x) \bigg)$.  Note that each ball must be one of the $b$, $c$, $d$ or $e$-type, we have a covering
\begin{equation}
     B_r \subset  \big(\cC^1_0 \cup \cN^1)  \cup \bigcup_{b \in B^1} B_{r_b}(x_{b}) \cup \bigcup_{c \in C^1} B_{r_c}(x_{c}) \cup \bigcup_{d \in D^1} B_{r_d}(x_{d}) \cup \bigcup_{e \in E^1} B_{r_e}(x_{e}).
\end{equation}

Now for each $c$-ball in $C^1$ group, we apply the same decomposition to obtain the refined covering 
\begin{equation}
     B_r \subset  \big(\cC^2_0 \cup \cN^2)  \cup \bigcup_{b \in B^2} B_{r_b}(x_{b}) \cup \bigcup_{c \in C^2} B_{r_c}(x_{c}) \cup \bigcup_{d \in D^2} B_{r_d}(x_{d}) \cup \bigcup_{e \in E^2} B_{r_e}(x_{e}).
\end{equation}

By iterating this decomposition, finally we will obtain a covering
\begin{equation}
     B_r \subset  \big(\cC_0 \cup \cN)  \cup \bigcup_{b \in B} B_{r_b}(x_{b}) \cup \bigcup_{d \in D} B_{r_d}(x_{d}) \cup \bigcup_{e \in E} B_{r_e}(x_{e}),
\end{equation}
as the group of $C$-balls will converge to a subset of $\cC_0$. The proof that $\cN$ is indeed a neck region is standard and is omitted here. See \cite{CJN} for more details. 

By Neck structure Theorem \ref{t:Neck_Struture}, we have
\begin{equation}
    \sum_{b \in B} r_b^k + \sum_{d \in D} r_d^k + \sum_{b \in E} r_e^k + \cH^k(\cC_0) \leq C(n)r^k.
\end{equation}

The proof is now finished.
\end{proof}

By applying the covering arguments to $d$-balls and $c$-balls iteratively, we will end up with a covering of without any $c$-ball or $d$-ball. By this we can finish the proof of Neck decomposition theorem \ref{t:Neck_decompose}.

\begin{proof}{(Neck Decomposition Theorem \ref{t:Neck_decompose})}
 First we can cover the unit ball $B_1$ using $B_r(x)$ with $r$ sufficiently small that each $B_r$ is a $(C(n,\lambda)\Lambda, C_2(\lambda))$-good ball, i.e. $\sup_{y \in B_r(x)} D(y,s) \leq 
    C\Lambda+\epsilon$ for any $s \leq C_2(\lambda) r$.  Applying the decomposition of $c$-ball and $d$-ball iteratively, by choosing $\tau$ small enough, we have the following covering    \begin{equation}
     B_1 \subset  \bigcup_{a} \big(\cC_{0,a} \cup \cN_a \cap B_a(x_a) \big)  \cup \bigcup_{b } B_{r_b}(x_{b}) \cup \bigcup_{e} B_{r_e}(x_{e}) \cup S_d,
\end{equation}
where $\cH^k(S_d) = 0$ and $\sum_a r^k_a + \sum_b r^k_b + \sum_e r^k_e + \cH^k(\bigcup_{a} \cC_{0,a}) \leq C(n, \lambda, \alpha, \Lambda, \epsilon, \eta)$.

By definition of $e$-ball, we have $D(x, s) < d -\epsilon$ for some $s \geq 10^{-3}r_e$ for any $x \in B_{r_e}(x_e)$. By Lemma \ref{l:DI_Drop_Elliptic}, $D(x, s) \leq d+1+\epsilon$ for any $x \in B_{r_e}(x_e)$. Hence we can cover $B_{r_e}(x_e)$ using Vitali covering of radius $\epsilon^{C(\alpha)}$ and apply the previous constructions to each such $(d-1)$-good ball. We continue this process for at most $C(n,\lambda)\Lambda$-times, and finally we obtain the covering
\begin{equation}
     B_r \subset  \bigcup_{a} \big(\cC_{0,a} \cup \cN_a \cap B_a(x_a) \big)  \cup \bigcup_{b } B_{r_b}(x_{b}) \cup  S_0,
\end{equation}
where $S_0 \equiv \bigcup_{1 \leq i \leq C\Lambda}S_i$ and thus $\cH^k(S_0) = 0$. 

Finally the fact that the quantitative stratum $\cS^k_{\delta} \subset S^{k, \epsilon, \eta}$ provided $\eta$ and $\epsilon$ is chosen small  follows directly from cone splitting Lemma. Indeed, by Corollary \ref{c:finite_non_pinch}, there exists some small $r$ such that the doubling index is pinched at the scale around $r$. Then by Theorem \ref{t:Quan_Unique}, $u$ is almost $0$-symmetric in $B_r(x)$. Further we have $u$ is almost $(k+1)$-symmetric since $B_{2 r_b}(x_b)$ is almost $(k+1)$-symmetric. Note that here we need to choose small $\eta$ and $\epsilon$ depending on $\delta$. This proves that $x \notin \cS^k_{\delta}$. We can apply similar arguments to show that $x \notin \cS^k_{\delta}$ provided $x \in \cN_a \cap B_a(x_a)$. This finishes the proof.
\end{proof}

\begin{proof}[Proof of Theorem \ref{t:Estimate_on_Strata}]
The proof now follows from Neck Decomposition Theorem \ref{t:Neck_decompose}, cone-splitting Lemma \ref{l:Neck_cone_splitting}. The rest of the arguments will be the same as the proof of Theorem 1.7 in \cite{CJN}.    
\end{proof}


\begin{thebibliography}{HHHH}

\bibitem {ABR} Axler S., Bourdon P., and Ramey W., Harmonic function theory, second ed., Graduate
Texts in Mathematics 137, Springer-Verlag, New York, 2001.

  \bibitem {Alm} Almgren, Jr., Dirichlet's problem for multiple valued functions and the regularity of mass minimizing integral currents, in "Minimal Submanifolds and Geodesies," (M. Obata, Ed.), pp. 1-6, North-Holland, Amsterdam.

 \bibitem {BeLin} Bellova,K., Lin,F.-H.: Nodal sets of Steklov eigenfunctions. Calc.Var.PDE 54,2239-2268 (2015)

 \bibitem {Bru} Br\"uning,J.: \"Uber Knoten won Eigenfunktionen des Laplace-Beltrami-operators. Math.Z.158,15-21 (1978)

 \bibitem {CJN} Cheeger, J.; Jiang, W.; Naber, A., Rectifiability of singular sets of noncollapsed limit spaces with Ricci curvature bounded below. Ann. of Math. (2) 193 (2021), no. 2, 407-538.

 \bibitem {CN} Cheeger, J.; Naber, A.;  Lower bounds on Ricci curvature and quantitative behavior of singular sets, Invent. Math. 191 (2013), 321-339.

 \bibitem {CNV} Cheeger, J.; Naber, A.; Valtorta, D. Critical sets of elliptic equations. Comm. Pure Appl. Math. 68 (2015), no. 2, 173-209

 \bibitem {CM} Colding, T.H., Minicozzi II, W.P.: Lower bounds for nodal sets of eigenfunctions. Commun. Math.
Phys. 306, 777-784 (2011)

 \bibitem {Dong} Dong,R.-T.:Nodal sets of eigenfunctions on Riemann surfaces. J.Differ.Geom.36,493-506(1992)

 \bibitem {DF1} Donnelly, H., Fefferman, C.: Nodal sets of eigenfunctions on Riemannian manifolds. Invent. Math. 93(1), 161183 (1988)

 \bibitem {DF2} Donnelly,H.,Fefferman,C.:Nodal sets of eigenfunctions: Riemannian manifolds with boundary.In: Analysis, Et Cetera. Academic Press, Boston, pp. 251-262 (1990)

 \bibitem {DF3} Donnelly,H.,Fefferman,C.:Nodal sets for eigenfunctions of the Laplacian on surfaces. J.Am.Math. Soc. 3(2), 333-353 (1990)

 \bibitem {GL} Garofalo, N., Lin, F.-H.: Monotonicity properties of variational integrals, $A_p$ weights and unique continuation. Indiana Univ. Math. 35, 245-268 (1986)

 \bibitem {GT} Gilbarg D. and Trudinger N. S., Elliptic partial differential equations of second or-
der, Classics in Mathematics, Springer-Verlag, Berlin, 2001, Reprint of the 1998 edition.

 \bibitem {Hansingular} Han, Q., Singular sets of solutions to elliptic equations, Indiana Univ. Math. J. 43 (1994), 983-1002.

 \bibitem {Hanhamornic} Han, Q., Nodal sets of harmonic functions. Pure Appl. Math. Q. 3 (2007), no. 3, Special Issue: In honor of Leon Simon. Part 2, 647-688.

 \bibitem {HLJPDE} Han, Q.; Lin, F.-H., On the geometric measure of nodal sets of solutions. J. Partial Differential Equations 7 (1994), no. 2, 111-131. 

 \bibitem {HLparobolic} Han, Q.; Lin, F.-H. Nodal sets of solutions of parabolic equations. II. Comm. Pure Appl. Math. 47 (1994), no. 9, 1219-1238.

  \bibitem {HLbook} Han, Q., Lin, F.-H.: Nodal Sets of Solutions of Elliptic Differential Equations, book in preparation (online at http://www.nd.edu/qhan/nodal.pdf)
 
 \bibitem {HHL} Han, Q.; Hardt, Robert; Lin, Fanghua Geometric measure of singular sets of elliptic equations. Comm. Pure Appl. Math. 51 (1998), no. 11-12, 1425–1443.

 \bibitem {HHLhigher} Han, Q.; Hardt, R.; Lin, F.H., Singular sets of higher order elliptic equations. Comm. Partial Differential Equations 28 (2003), no. 11-12, 2045-2063. 
 
 \bibitem {HHHN} Hardt, R.; Hoffmann-Ostenhof, M.; Hoffmann-Ostenhof, T.; Nadirashvili, N. Critical sets of solutions to elliptic equations. J. Differential Geom. 51 (1999), no. 2, 359-373.
 
  \bibitem {HS} Hardt,R.,Simon,L.:Nodal sets for solutions of elliptic equations. J. Differ. Geom.30,505-522(1989)

 \bibitem {HSo} Hezari, H., Sogge,C.D., A natural lower bound for the size of nodal sets. Anal. PDE5(5),1133-1137
(2012)

 \bibitem {HHN} Hoffmann-Ostenhof, M.; Hoffmann-Ostenhof, T.; Nadirashvili, N. Critical sets of smooth solutions to elliptic equations in dimension 3. Indiana Univ. Math. J. 45 (1996), no. 1, 15-37.

\bibitem {JN}  Jiang W., Naber A. $L^2$ curvature bounds on manifolds with bounded Ricci curvature," Annals of Mathematics, Ann. of Math. (2) 193(1), 107-222, 2021

 \bibitem {Linconj}  Lin,F.-H.:Nodal sets of solutions of elliptic equations of elliptic and parabolic equations.Commun. Pure Appl. Math. 44, 287-308 (1991)
 
 \bibitem {LLbetti}  Lin, F.H.; Liu, D., On the Betti numbers of level sets of solutions to elliptic equations. Discrete Contin. Dyn. Syst. 36 (2016), no. 8, 4517-4529. 
 
 \bibitem {LS1}  Lin, F.H.; Shen, Z., Nodal sets and doubling conditions in elliptic homogenization. Acta Math. Sin. (Engl. Ser.) 35 (2019), no. 6, 815-831.

  \bibitem {LS2} Lin, F.H.; Shen, Z., Critical Sets of Solutions of Elliptic Equations in Periodic Homogenization, arXiv:2203.13393v1 [math.AP] , 2022.
 
 \bibitem {LZhu}  Lin, F.H.; Zhu, J., Upper bounds of nodal sets for eigenfunctions of eigenvalue problems. Math. Ann. 382 (2022), no. 3-4, 1957-1984.

 \bibitem {Loupper} Logunov, A.: Nodal sets of Laplace eigenfunctions: polynomial upper estimates of the Hausdorff measure. Ann. Math. 187, 221-239 (2018)

 \bibitem {Lolower} Logunov,A.:Nodal sets of Laplace eigenfunctions:proof of Nadirashvili's conjecture and of the lower bound in Yau’s conjecture. Ann. Math. 187, 241-262 (2018)

 \bibitem {LM} Logunov,A.,Malinnikova,E.:Nodal sets of Laplace eigenfunctions: estimates of the Hausdorff measure in dimension two and three, 50 years with Hardy spaces, pp. 333–344, Oper. Theory Adv. Appl., vol. 261. Birkh\"auser/Springer, Cham (2018)

 \bibitem {LM1} Logunov, A.; Malinnikova, E., Review of Yau's conjecture on zero sets of Laplace eigenfunctions. Current developments in mathematics 2018, 179–212, Int. Press, Somerville, MA, [2020],

 \bibitem {LMNN} Logunov, A.; Malinnikova, E.; Nadirashvili, N.; Nazarov, F. The sharp upper bound for the area of the nodal sets of Dirichlet Laplace eigenfunctions. Geom. Funct. Anal. 31 (2021), no. 5, 1219-1244.


 \bibitem {Man} Mangoubi, D.:A remark on recent lower bounds for nodal sets. Commun. Partial Differ. Equ.36(12), 2208-2212 (2011)

 \bibitem {Mil}Milnor, J. On the Betti numbers of real varieties. Proc. Amer. Math. Soc. 15 (1964), 275-280.
\bibitem{NVYM}Naber, A.; Valtorta, D., "Energy identity for stationary Yang Mills," Invent. Math., vol. 216, iss. 3, pp. 847-925, 2019.

 \bibitem {NV} Naber, A.; Valtorta D., Volume estimates on the critical sets of solutions to elliptic PDEs. Comm. Pure Appl. Math. 70 (2017), no. 10, 1835-1897.

\bibitem {N88}Nadirashvili, N.,  The length of the nodal curve of an eigenfunction of the Laplace operator. Uspekhi
Mat. Nauk, 43(4(262)):219-220, 1988.

\bibitem {N} Nadirashvilli, N., Geometry of nodal sets and multiplicity of eigenvalues, Current Developments in Mathematics, 1997, 231-235.


\bibitem {P} Plis, A. On non-uniqueness in Cauchy problem for an elliptic second order differential equation, Bull. Acad. Polon. Sci. Ser. Sci. Math. Astronom. Phys. 11 (1963), 95–100. MR 0153959.
Zbl 0107.07901.


 \bibitem {SWZ} Sogge,C.D., Wang,X., Zhu,J.:Lower bounds for interior nodal sets of Steklov eigenfunctions. Proc. Am. Math. Soc. 144(11), 4715-4722 (2016)

 \bibitem {SZ} Sogge,C.D.,Zelditch,S.:Lower bounds on the Hausdorff measure of nodal sets.Math.Res.Lett.18, 25-37 (2011)

 \bibitem {Ste} Steinerberger, S.: Lower bounds on nodal sets of eigenfunctions via the heat flow. Commun. Partial Differ. Equ. 39(12), 2240-2261 (2014)

 \bibitem {TZ} Toth, J., Zelditch, S., Counting nodal lines which touch the boundary of ananalytic domain.J.Differ. Geom. 81(3), 649-686 (2009)

 \bibitem {Yau} Yau, S.T.: Problem Section, Seminar on Differential Geometry, Annals of Mathematical Studies, vol. 102, pp. 669-706. Princeton University Press, Princeton (1982)

 \bibitem {Zhu2} Zhu, J.:Doubling inequality and nodal sets for solutions of bi-Laplace equations. Arch. Ration. Mech. Anal. 232(3), 1543-1595 (2019)





\end{thebibliography}
\end{document}